\documentclass{article}
\usepackage{amsmath}
\usepackage{amsfonts}
\usepackage{amssymb}
\usepackage{graphicx}
\usepackage{tikz,pgfplots}
\usetikzlibrary{shapes,snakes}
\usepackage{epstopdf}
\usepackage{color}%
\usepackage{subfigure}
\usepackage{enumerate}
\usepackage[hmargin=3cm,vmargin=3cm]{geometry}
\providecommand{\U}[1]{\protect\rule{.1in}{.1in}}
\providecommand{\U}[1]{\protect\rule{.1in}{.1in}}
\newtheorem{theorem}{Theorem}

\newtheorem{example}{Example}
\newtheorem{lemma}[theorem]{Lemma}

\newtheorem{remark}[theorem]{Remark}

\newenvironment{proof}[1][Proof]{\noindent\textbf{#1.} }{\ \rule{0.5em}{0.5em}}

\begin{document}
\stepcounter{section} 
\title{Numerical Solution of Nonlinear Abel Integral Equations: An $ hp $-Version Collocation Approach}
\author{Raziyeh Dehbozorgi\thanks{Institute for Advanced Studies in Basic Sciences,
		Zanjan, Iran e-mail:
		\texttt{r.dehbozorgi2012@gmail.com} } \and
	Khadijeh Nedaiasl\thanks{Institute for Advanced Studies in Basic Sciences,
Zanjan, Iran, e-mail: \texttt{nedaiasl@iasbs.ac.ir  \&   knedaiasl85@gmail.com} }}
\maketitle
\begin{abstract}
	This paper is concerned with the numerical solution  for a class of  nonlinear weakly singular Volterra integral equation of the first kind.  The existence and uniqueness issue of the nonlinear Abel integral equations is studied completely. An $hp$-version  collocation method in conjunction with Jacobi polynomials is introduced so as an appropriate  numerical solution to be found. 
	 We analyze it properly and find an error estimation in $L^2$-norm. The efficiency  of the method is illustrated by some numerical experiments. 
\end{abstract}
\vspace{1em} \noindent\textbf{Keywords:} {nonlinear operator, first kind Volterra integral equation, weakly singular operator,  $hp$-version collocation method, error analysis.  }

\vspace{1em} \noindent\textbf{2010 Mathematics Subject Classification: 45H30; 45D05; 65L60; 65L70. }

\section{Introduction}
This paper deals with the numerical solution of the nonlinear weakly  singular Volterra integral equation of the first kind 
\begin{equation}\label{asli}
\mathcal{K}u(t)	:=\int_{0}^{t} (t-s)^{\alpha-1}\kappa(t,s) \psi(s,u(s))\mathrm{d}s=f(t), \quad  0<\alpha\leq 1, \quad 0\leq t\leq T < \infty,
\end{equation}
 where $k(t,s)$, $\psi(s,u(s))$ in the kernel and $f(t)$ the right-hand side term are  known and $u(t)$ is the unknown to be determined.
This equation could be expressed by an operator notation in a more  general form 
\begin{equation}
	\mathcal{K}u(t):=\int_{0}^{t} (t-s)^{\alpha-1}\kappa(s,t,u(s))\mathrm{d}s=f(t), \quad  t \in \Omega:=[0,T],
\end{equation}
where the operator $ \mathcal{K}u $ is nonlinear Abel integral operator. 

 From Niels Henrik Abel's endeavor to generalize the tautochrone problem  till today applications of Abel integral equations such as
the practical physical models originating  from
 spectroscopy, astrophysics (cf. \cite{gorenflo91}, \cite{TTot}, \cite{von2014bayesian})
 and  inverse problems arising in image reconstruction \cite{engl1996regularization}, there is a long way which signifies the importance of these equations. In addition, 
 the fractional differential (FD) operators in one dimensional domains are mainly defined by these   operators;  so,  the surge of FD equations 
 in research and application  is a leading  motivation for further research in the numerical solutions of the generalized Abel integral equations. 
 
 In the literature, the majority of the existing numerical approaches deal with the smooth kernel in linear form or the especial case $\alpha=\frac{1}{2}.$ Moreover,  in order to obtain an efficient approximation, one should increase the number of  mesh points ($ h$-version) or the degree of polynomials in the expansion ($p$-version).
In order to employ both beneficial features of $h$- and $p$- versions simultaneously,  we investigate the $hp$-version collocation method to achieve an appropriate solution 
for nonlinear Abel integral equations. For this aim, it is necessary to  express the existence and uniqueness of the solution  which is investigated properly in this paper. 

Due to efficiency and accuracy, the $hp$-version Galerkin and collocation methods have received considerable attentions. For example, the $ hp $-version of discontinuous Galerkin and Petrov-Galerkin have been studied for integro-differential equations of Volterra types ( for more details see \cite{mustapha,MR3266490}).
Sheng et al. have introduced a multi-step Legendre-Gauss spectral collocation method and given a comprehensive analysis of convergence in $ L^2$-norm  for the nonlinear Volterra integral equations of the second kind \cite{sheng}. This approach has been extended to the Volterra integral and integro-differential equations with vanishing delays \cite{MR3434874,MR3488071}. Locally varying time steps makes these methods popular for investigating the numerical solution of integral equations with weakly singular kernels \cite{MR3673690}. 
Recently, the $ hp$-version collocation method is studied for nonlinear Volterra integral equation of the first kind \cite{nedai} which is our starting point to develop the results for the nonlinear Abel integral equations.

\subsection{Relevant works}\label{rel}
Abel integral equations are of great importance  due to the presence of the singularity and being as an inverse problem. Aside these difficulties, nonlinear form of these equations needs more attention.
 There are some studies on the existence and uniqueness of the solution for Abel type integral equations. In Section 1.5 of \cite{brunner2017volterra}, Brunner seeks a unique solution in the space of continuous functions on $\Omega$ for Eq. \eqref{asli} for the linear case. In the lecture  note by Gorenflo and  Vessella \cite{gorenflo91}, the existence and uniqueness issue of the nonlinear Abel integral equations of the second kind is studied. Furthermore, some results on the regularity of the solution for the first kind nonlinear Abel integral equations in the spaces $C(\Omega)$ and $L^2(\Omega)$ are reported in  \cite{branca1978} and \cite{ANG199463}. In this paper, we attempt to study the well-posedness of the problem in some weighted Sobolev spaces.

 The product integration  and  adaptive Huber methods can solve the linear case of Eq. \eqref{asli} locally \cite{cameron84,bieniasz08}.  
 A Nystr\"{o}m-type	method  based on the trapezoidal and composite trapezoidal rules is analyzed for the Abel integral equation in \cite{eggermont1981new} and \cite{plato12}, respectively. The aforementioned schemes are utilized for linear Abel integral equations. 
In spite of the abundant research concerned with the numerical analysis of the linear second-kind of weakly singular integral equations \cite{eggermont1981new,plato12,cameron84,bieniasz08}, the numerical analysis of  nonlinear Abel integral equations is scarce. Due to the weak singularity and nonlinearity, the numerical methods for this class of the integral equations are less dealt with, for instance two specific research on this subject are \cite{ANG199463, branca1978}. 	  
	In \cite{branca1978}, Branca utilizes the interpolation quadrature technique with linear and quadratic polynomials in order to approximate the nonlinear Abel integral equation.  
Here, we develop the idea of interpolation quadrature technique in conjunction with global methods to approximate the numerical solution of nonlinear Abel integral equation in an efficient way. 

\subsection{Our contribution} 
	An important aspect of this method is its flexibility with respect to the step size and the order of polynomials in each sub-interval. As we will identify in the numerical experiments, the proposed collocation method works well for the approximation of the equations with non-smooth solutions. The main features of this paper are as follows: 
	\begin{description}
		\item[1. ] Developing the $hp$-version collocation method for the weakly singular integral equations of the first kind is one of the aspects of the present manuscript. 
		In accordance with the heuristic of the scheme which converts the first kind integral equation into the second one, without any restriction on the kernel, the discretized equation becomes more regular.  
		Not only does the scheme overcome the difficulty of the nonlinear term in these equations which are rarely investigated, but also the local view point of scheme makes it a powerful tool for better approximating, especially when the unknown solution is non-smooth.   
		
		\item [2. ] The  Jacobi polynomials are utilized to derive an exact quadrature formula for the integrals with weakly singular integrand which can be accounted as a merit of these polynomials.  Being adjustable, the parameters $M$ and $N$, which are related to $h$- and $p$-versions, cause  the scheme to be  more accurate and applicable. 	
		Furthermore, the error estimation for the presented schemes is  analyzed in the sense  of $L^2$-norms and the numerical results are compatible with these findings.

		\item[3. ] Investigating the numerical solution for various $\alpha$'s, non-smooth kernel functions $\kappa$ and the right hand-side functions $f(t)$ makes this scheme so conducive and persuasive regarding the real-word problems. 
		Moreover, in order to reduce the computational cost, we introduce an adaptive algorithm in the approximation procedure to show the efficiency of the method well (for more details, see Remark \ref{rem3}).
\end{description}
	
	This paper is organized in the following way. In Section \ref{sec 2} we give some regularity results for the nonlinear Abel integral equation of the first kind. Section \ref{prelim} is devoted to the description of the $ hp$-version collocation method for the first kind weakly singular nonlinear Volterra integral equation. In Section \ref{error analysis}, an error analysis of the proposed method is provided in some suitable Hilbert spaces. 
	Finally, in order to show the applicability and efficiency of the method and compare it with other methods, several examples with smooth and non-smooth solutions are illustrated in Section \ref{numerical experiments}.

\section{Theoretical treatment of the problem}
\label{sec 2} In order to find out a suitable numerical scheme for the solution of Eq. (\ref{asli}), knowledge of the behavior of the exact solution is important. The kernel of the integral equation  contains a singular term, for the reason that weighted Lebesgue spaces are   utilized as  the suitable functional spaces. For this aim, 
let us define the weight function $ \chi^{\alpha, \beta}(x) :=(1-x)^{\alpha}(1+x)^{\beta} $ on the interval $ \Lambda :=[-1,1] $ for $ \alpha, \beta >-1 $. For $ r \in \mathbb{N}$, $ H^{r}_{\chi^{\alpha,\beta}}(\Lambda) $ is a weighted Sobolev space defined by
\[
H^{r}_{\chi^{\alpha,\beta}}(\Lambda) = \Big\{ v ~|~ v ~\text{is measurable and } \Vert v \Vert_{r, \chi^{\alpha, \beta}} < \infty \Big\},
\]
where  
\[
\Vert v \Vert_{r, \chi^{\alpha, \beta}} = \Big( \sum_{k=0}^{r} \vert v \vert_{k, \chi^{\alpha, \beta}}^2 \Big)^\frac{1}{2}.
\]
The above semi-norm is defined as 
$
\vert v \vert_{k, \chi^{\alpha, \beta}} = \Vert \partial^{k}_{x}v \Vert_{\chi^{\alpha +r, \beta+r}},
$
where  $ \Vert . \Vert_{\chi^{\alpha, \beta}}  $ is an appropriate norm for the space $ L^{2}_{\chi^{\alpha, \beta}}(\Lambda)$. 
For arbitrary real number $  r=[r]+\theta $ with $ \theta\in (0,1)$, $ H^{r}_{\chi^{\alpha,\beta}}(\Lambda) $ can be defined by the interpolation space  as \[H^{r}_{\chi^{\alpha,\beta}}(\Lambda)=[H^{[r]}_{\chi^{\alpha,\beta}}(\Lambda),H^{[r]+1}_{\chi^{\alpha,\beta}}(\Lambda)]_\theta. \]
More details can be seen in  \cite{lilian,lions2012non}. 

We need some definitions from fractional calculus. The Riemann-Liouville integral operator $_{0}\mathcal{I}_{x}^{r}$ is defined as follows
	\begin{equation}
	_{0}\mathcal{I}_{x}^{r}u(x) = \int_{0}^{x}(x-t)^{r -1}u(t)\mathrm{d}t,
	\end{equation}
	and  $ _0 D_x^ r$,	the Riemann-Liouville fractional derivative of order $ r $ for a function $u \in H^n(\Omega)$ can be defined as
	\begin{equation}
		_{a}D_x^r u=D^{n}\, _{0}\mathcal{I}_{x}^{n -r}u,
	\end{equation}
	where the operator $D^{n}$ denotes the classical derivative of order $n$ \cite{diethelm2010analysis}.

 As discussed in \ref{rel}, the well-posedness of the problem \eqref{asli} have been investigated in the spaces $C(\Omega)$ and $L^2(\Omega)$ by Branca, Ang and Gorenflo  \cite{branca1978,ANG199463}.  
In the following theorem, some adequate assumptions are given in order to have a unique solution for \eqref{asli} in the weighted Sobolev spaces  $H^{^{m-1}}_{\chi^{\alpha-1,0}}(\Omega)$. 

\begin{theorem}\label{thm1}
	Assume that the Eq. \eqref{asli} satisfies the following assumptions
	\begin{enumerate}[i.]
		\item $f(t)\in H^{^m}_{\chi^{\alpha-1,0}}(\Omega), ~ f(0)=0$, \label{item1}
		\item $\kappa(s,t) \in C^{^{m}}(\Omega\times \Omega)$ and $ \kappa(t,t)\neq 0$ $\text{for all } t \in \Omega$, \label{item2}
		\item $\psi(s,u)\in H^{m-1}_{\chi^{\alpha-1,0}}(\Omega\times \mathbb{R})$,\label{item3}
		\item $\inf \Big\{  \vert \frac{\partial \psi }{\partial u}(s,u) \vert \ \big| (s,u) \in \Omega \times \mathbb{R}\Big\} \geq M>0$, \label{item4}
		\item $\psi (s,u)$ is Lipschitz continuous w.r. to $ u$, \label{item5}
		\item let $k(t)= \int_{0}^t \int_x^t (t-y)^{-\alpha}(y-x)^{\alpha-1}\kappa(y,x)\psi(x,u(x))\,\mathrm{d}y \,\mathrm{d}x$, then $k\in H^{^{m-1}}_{\chi^{\alpha-1,0}}(\Omega)$.  
		\label{item6}
	\end{enumerate}
	Then it has a unique solution $ u $ in $H^{^{m-1}}_{\chi^{\alpha-1,0}}(\Omega)$.
\end{theorem}
\begin{proof}
	In advance, let us convert the main problem \eqref{asli} to the second kind integral equation by multiplying both sides into $ (x-t)^{-\alpha} $ and taking integration; hence, from the assumption  (\textit{\ref{item1}}), Eq. \eqref{asli} reads that 
	\begin{equation}\label{sec}
	\psi(t,u(t))+\int_0^t \dfrac{\mathcal{L}_t(t,s,u(s))}{\kappa(t,t)}\mathrm{d}s=\dfrac{(_0 D_x^\alpha f)(t)}{\kappa(t,t)},
	\end{equation}     
	where 
	\[\mathcal{L}(t,x,u(x))=\int_x^t (t-y)^{-\alpha}(y-x)^{\alpha-1}\kappa(y,x)\psi(x,u(x))\,\mathrm{d}y.\]
	By the new variable $y=x+\tau(t-x)$, the operator $ \mathcal{L} $ could be written as 
	\[
\mathcal{L}(t,x,u)=	\int_{0}^{1} (1-\tau)^{-\alpha}\tau^{\alpha-1}\kappa(x+\tau(t-x),x)\psi(x,u(x))\mathrm{d}\tau.
	\]
	It is evident that $ \mathcal{L}(x,x,u)=\kappa(x,x)\psi(x,u) $  and 
	\begin{equation}
	\mathcal{L}_{t}(t,x,u) =\psi(x,u(x))\, k^{*}(t,x),
	\end{equation} 
	is continuous  on $ \Omega\times \mathbb{R} $ and  Lipschitz continuous with respect to $ u $ with the same constant as for $ \psi $ and $ k^{*}(t,x) = \int_{0}^{1} (1-\tau)^{-\alpha}\tau^{\alpha}\kappa_{t}(x+\tau(t-x),x)\mathrm{d}\tau $.
Conditions (\textit{\ref{item1}})-(\textit{\ref{item3}}) lead that each function $ u(t)$ is a solution of Eq. \eqref{sec} if and only if it is a solution of Eq. \eqref{asli}.
	In order to prove the existence of a solution for Eq. \eqref{sec}, we trace \cite{MR847018} and  define the sequence $ \{u_n(t)\}_{n\in \mathbb{N}}$ as follows:
	\begin{equation}
	\begin{split}
	\psi(0,u_0(t))&:=\frac{(_0 D_x^\alpha f)(0)}{\kappa(0,0)},\\
	\psi(t,u_{n+1}(t))&:=\frac{(_0 D_x^\alpha f)(t)}{\kappa(t,t)}-\int_{0}^{t} \dfrac{\mathcal{L}_t(t,s,u_{n}(s))}{\kappa(t,t)}\mathrm{d}s, \quad n\geq 1.
	\end{split}
	\end{equation} 
	By the assumptions (\textit{\ref{item3}})  and (\textit{\ref{item4}}), the  function $ \psi(t,u(t)) $ is strictly monotonic continuous
	with respect to $ u $. So by considering the Inverse Theorem \cite[p. 68]{courant}, $ u_0 $ is well-defined and belongs to  $ H^{^{m-1}}_{\chi^{\alpha-1,0}}(\Omega)$.
	Now, using induction hypothesis,  $ u_n $ is well-defined and belongs to $  H^{^{m-1}}_{\chi^{\alpha-1,0}}(\Omega). $
	From the assumption (\textit{\ref{item1}}), it is deduced that $f'\in H^{^{m-1}}_{\chi^{\alpha-1,0}}([0,x])$. On the other hand, from the fractional calculus we get that 
	\[
	(_0 D_x^\alpha f)(x)=\frac{\mathrm{d}}{\mathrm{d}x}\int_{0}^{x}\frac{f(t)}{(x-t)^{\alpha}}\mathrm{d}t= \int_{0}^{x}\frac{f'(t)}{(x-t)^{\alpha}}\mathrm{d}t+\frac{x^{-\alpha}}{\Gamma(1-\alpha)}f(0).
	\]
	 Now, utilizing $f(0)=0$ and Theorem 3.1 of the paper \cite{jin2015}, it is concluded that  $\int_{0}^{x}\frac{f'(t)}{(x-t)^{\alpha}}\mathrm{d}t \in H^{^{m-\alpha}}_{\chi^{\alpha-1,0}}(\Omega)$, which is a subset of $ H^{^{m-1}}_{\chi^{\alpha-1,0}}(\Omega) $.
	From the above argument and the assumption  (\textit{\ref{item6}}), we deduce that the function
	\[ \frac{(_0 D_x^\alpha f)(t)}{\kappa(t,t)}-\int_0^t \frac{\mathcal{L}_t(t,s,u_{n}(s))}{\kappa(t,t)}\mathrm{d}s, \]
	belongs to $  H^{^{m-1}}_{\chi^{\alpha-1,0}}(\Omega). $ Hence, by the Inverse Theorem,  $ u_{n+1}\in H^{^{m-1}}_{\chi^{\alpha-1,0}}(\Omega) $. 
	Using the assumptions (\textit{\ref{item4}}) and (\textit{\ref{item5}})
	one can conclude that 
	\[\vert u_{n+1}(t)-u_n(t)\vert \leq \Big(\frac{JL}{M}\Big)^n\frac{t^n}{n!}\max\limits_{s \in \Omega} \vert u_1(s)-u_0(s)\vert,\]
	where $ L $ is the Lipschitz constant in the assumption (\textit{\ref{item5}}) and $J:=\max \big\{ \vert \frac{k^*(t,s)}{k(t,t)} \vert ~\big|~ (t,s) \in \Omega \times \Omega \big\}$. Therefore, without loss of generality for $ m>n,$
	\[\vert u_{m}(t)-u_n(t)\vert \leq \sum_{i=n}^{m-1}\vert u_{i+1}(t)-u_i(t)\vert\leq \Vert u_1(t)-u_0(t)\Vert_{\infty}\sum_{i=n}^{m-1}\Big(\frac{JL T}{M}\Big)^{^i}\frac{1}{i!}.\]
	The term $\sum\limits_{i=0}^{\infty}(\frac{JL T}{M})^{^i}\frac{1}{i!}  $ is convergent, so the Cauchy sequence $ \lbrace u_n\rbrace $ is convergent uniformly to
	\[\lim\limits_{n\rightarrow \infty}u_n(t)=u(t),\]
	where $ u(t) $ belongs to $  H^{^{m-1}}_{\chi^{\alpha-1,0}}(\Omega).$ This result follows from the fact that  $ u_n(t)\in   H^{^{m-1}}_{\chi^{\alpha-1,0}}(\Omega)$.
\end{proof}
\section{Numerical scheme}\label{prelim}
In this section, we introduce an $hp$-version Jacobi collocation method for nonlinear weakly singular integral equations of the first kind. In order to get a self-contained paper, some basic properties of the shifted Jacobi-Gauss  and Legendre-Gauss-Lobatto polynomial interpolations are introduced in the following subsection.
\subsection{Preliminaries}
{\bf The shifted Jacobi-Gauss interpolation operator.}   
Let us denote the standard Jacobi polynomial of degree $k$ by $J^{\alpha, \beta}_{k}(x)$, for $\alpha, \beta > -1$.  It is well-known that the set of Jacobi polynomials makes a complete orthogonal system with respect to the  weight function $\chi^{\alpha, \beta}(x)$ which means that 
\begin{equation}\label{orth}
\int_{\Lambda}J^{\alpha, \beta}_{k}(x)J^{\alpha, \beta}_{j}(x)\chi^{\alpha, \beta}(x)\mathrm{d}x = \gamma_{k}^{\alpha, \beta}\delta_{k,j},  
\end{equation} 
wherein $\delta_{k,j}$ is the Kronecker function, and 
\[ \gamma^{\alpha, \beta}_{k}=\begin{cases} 
\frac{2^{\alpha+\beta+1}\Gamma(\alpha+1)\Gamma(\beta+1)}{\Gamma(\alpha+\beta+2)}, & k= 0, \\
\frac{2^{\alpha+\beta+1}}{2k+\alpha+\beta+1}\frac{\Gamma(k+\alpha+1)\Gamma(k+\beta+1)}{k!\Gamma(k+\alpha +\beta+1)}, & k \geq 1. 
\end{cases}
\]
In order to work with these polynomials on the sub-intervals $\Omega_n$ properly, the shifted Jacobi polynomial of degree $k$  is also defined as follows
\begin{equation}
J^{\alpha, \beta}_{n,k}(t)=J_{k}^{\alpha, \beta}(\frac{2t-t_{n-1}-t_{n}}{h_{n}}), \quad t \in \Omega_n, \quad k\geq 0. 
\end{equation}
It is worth to mention that the set of shifted Jacobi polynomials constructs a complete orthogonal system with the weight function  $\chi_{n}^{\alpha, \beta}(t)= (t_{n}-t)^\alpha(t-t_{n-1})^{\beta}$. Similar to the relation \eqref{orth},  we can write 
 \begin{equation}\label{jabc}
 \int_{\Omega_n}J_{n,k}^{\alpha, \beta}(t)J_{n,j}^{\alpha, \beta}(t)\chi_{n}^{\alpha, \beta}(t)\mathrm{d}t
= (\frac{h_{n}}{2})^{\alpha+\beta+1} \gamma_{k}^{\alpha,\beta}\delta_{k,j}. 
\end{equation}
Let $x^{\alpha, \beta}_{n,j}$ be the zeros of the standard Jacobi polynomial of degree $k$ for $0\leq j \leq M_{n}$ and  $\omega^{\alpha, \beta}_{n,j}$ be the corresponding Christoffel numbers. Then we can define the shifted Jacobi-Gauss quadrature points on the interval $\Omega_n$ as follows
\begin{equation}
t^{\alpha, \beta}_{n,j}=\frac{1}{2}(h_{n}x_{n,j}^{\alpha, \beta}+t_{n-1}+t_{n}), \quad 0\leq j \leq M_{n}. 
\end{equation}
Let $\mathcal{P}_{M}(\Omega)$ be the set of all polynomials of degree at most $M$ on $\Omega$. It is known  from \cite{lilian,guo} that for any $\phi (t) \in \mathcal{P}_{2M_{n}+1}(\Omega_n)$ 
\begin{equation}\label{jab}
\int_{\Omega_n}\phi(t)\chi_{n}^{\alpha, \beta}(t)\mathrm{d}t = (\frac{h_{n}}{2})^{\alpha+\beta+1}\sum_{j=0}^{M_{n}}\phi(t^{\alpha, \beta}_{n,j})\omega_{n,j}^{\alpha, \beta},
\end{equation}
which leads to the result
\begin{equation}\label{jac3}
\sum_{j=0}^{M_{n}}J^{\alpha, \beta}_{n,p}(t^{\alpha, \beta}_{n,j})J^{\alpha, \beta}_{n,q}(t^{\alpha, \beta}_{n,j})\omega_{n,j}^{\alpha, \beta}=\gamma_{p}^{\alpha, \beta} \delta_{p,q}.
\end{equation}
For any $v \in C(\Omega_n)$,  the shifted Jacobi-Gauss interpolation operator in the $t$-direction is defined  as follows
\begin{equation}\label{inter}
\mathcal{I}_{t, M_{n}}^{\alpha ,\beta}v(t^{\alpha, \beta}_{n,j})=v(t^{\alpha, \beta}_{n,j}), \quad 0\leq j \leq M_{n}, 
\end{equation}
 and the following lemma reports an upper bound for the interpolation by \eqref{inter}. 
\begin{lemma}(\cite{wang2017})\label{lem1}
	For any ${v}\in H^{^m}_{\chi_n^{\alpha,\beta}}(\Omega_n)$ with integer $1\leq m \leq M_n+1$ and $\alpha,\beta>-1,$ we get
	\[\Vert {v}-\mathcal{I}_{x,M_n}^{\alpha,\beta} {v}\Vert_{\chi_n^{\alpha,\beta}}\leq c\, \sqrt{\dfrac{\Gamma(M_n+2-m)}{\Gamma(M_n+2+m)}}\Vert \partial^m_x {v}\Vert_{\chi_n^{\alpha+m,\beta+m}}.\]
 In particular, for any fixed $ m,$ we obtain 
	\[\Vert {v}-\mathcal{I}_{x,M_n}^{\alpha,\beta} {v}\Vert_{\chi_n^{\alpha,\beta}}\leq c (M_n+1)^{-m}\Vert \partial^m_x {v}\Vert_{\chi_n^{\alpha+m,\beta+m}} \leq c h_n^m (M_n+1)^{-m}\Vert \partial^m_x {v}\Vert_{\chi_n^{\alpha,\beta}}. \]
\end{lemma}
In the current paper, we are interested in the special case  when $\alpha:=\alpha-1$ and $\beta:=0$.
\\
{\bf The shifted Legendre-Gauss interpolation.}
Let $ L_p(t) $ be defined as
\begin{equation*}
	\begin{split}
		L_p(t)=\left\lbrace \begin{array}{lcc}
			l_p(t),&  t\in \Lambda,\\
			0,& \text{o.w},
		\end{array}\right.
	\end{split}
\end{equation*}
where $ l_p(t) $ is the Legendre polynomial of degree $p$. Therefore, the shifted Legendre polynomials of degree $ p $ over the subinterval $ \Omega_n $ are defined as
\[L_{n,p}(t)=L_p(\frac{2t-t_{n-1}-t_n}{h_n}), \quad t\in \Omega_n.\]
 In the previous definition about the  Jacobi interpolation operator and its properties, if we take $ \alpha=\beta=0 $, the shifted Legendre-Gauss interpolation $  \mathcal{I}^t_{M_n} $ can be defined as $\mathcal{I}^{0,0}_{t,M_n}$. For instance,  in Eqs. \eqref{jabc}, \eqref{jab} and \eqref{jac3}, we have 
 \begin{equation}\int_{\Omega_n}L_{n,p}(t)L_{n,q}(t)\mathrm{d}t=\frac{h_n}{2p+1}\delta_{p,q},\end{equation} 
\begin{equation}\label{hn}
\int_{\Omega_n}\phi(t)\mathrm{d}t=\frac{h_n}{2}\sum_{j=0}^{M_n}\phi(t_{n,j})w_{n,j},
\end{equation}
and
\begin{equation}\sum_{j=0}^{M_n}L_{n,p}(t_{n,j})L_{n,q}(t_{n,j})w_{n,j}=\frac{2}{2p+1}\delta_{p,q},\end{equation}
where  $\big \{t_{k,i},w_{k,i}\big\}_{i=0}^{M_k} $ are the shifted Legendre-Gauss quadrature nodes  and weights.
These shifted functions form a complete orthogonal set for $ L^2(\Omega_n)$ functions, i.e., for any function $ g\in L^2(\Omega_n) $, it can be represented as 
\[g(t)=\sum_{p=1}^{\infty}\hat{g}_{n,p}L_{n,p}(t).\]
Therefore,  the operator $  \mathcal{I}_{M_n}^t g(t)$ is stated by
\begin{equation}\label{app}
 \mathcal{I}_{M_n}^t g(t)=\sum_{p=0}^{M_n}\hat{g}_pL_p(t),
\end{equation}
where the coefficients $\hat{g}_{p}$ can be obtained by means of the orthogonality property of  Legendre polynomials as
\[ \hat{g}_{p}=\frac{2p+1}{2} \int_{\Omega_n}g(t)L_{n,p}(t)\mathrm{d}t.\]
  Utilizing the shifted Legendre-Gauss interpolation operator $  \mathcal{I}^t_{M_n} $ to approximate a function belonging to the space $H^{^m}(\Omega_n)$ leads to an error which is bounded as the following lemma shows.
\begin{lemma}(\cite{wang2017})\label{lemf}
For any ${v}\in H^{^m}(\Omega_n)$ with integer $1\leq m \leq M_n+1$, we get
\[\Vert {v}-\mathcal{I}^t_{M_n} {v}\Vert_{\Omega_n}\leq c h_n^m(M_n+1)^{-m}\Vert \partial^m_t {v}\Vert_{\Omega_n}.\]
\end{lemma}
\textbf{The shifted Legendre-Gauss-Lobatto interpolation.} Let $ \big\{x^L_{n,j},w^L_{n,j}\big\}_{j=0}^{M_n} $ be the nodes and Christoffel numbers of the
standard Legendre-Gauss-Lobatto interpolation on $ \Lambda$. The corresponding nodes of this interpolation on $ \Omega_n$ can be defined by
$ s_{n,j}^L=\dfrac{h_nx^L_{n,j}+t_n+t_{n-1}}{2}.$ For the definition of the shifted Legendre-Gauss-Lobatto interpolation $ \mathcal{I}_{s,M_n}^L $, it is easily observed that  $ \mathcal{I}_{s,M_n}^L v\in \mathcal{P}_{M_n}(\Omega_{n})$ and 
$ \mathcal{I}_{s,M_n}^L v(s_{n,j}^L)=v(s_{n,j}^L). $
 For any function $ \phi \in  \mathcal{P}_{{2M_n+1}}(\Lambda) $, the following identities are deduced from the main property of Legendre-Gauss-Lobatto quadrature,
\[\int_{\Omega_n}\phi(s)\mathrm{d}s=\frac{h_n}{2}\sum_{j=0}^{M_n} w^L_{n,j}\phi(s^L_{n,j}),\]
and
\[\sum_{j=0}^{M_n}L_{n,p}(s^L_{n,j})L_{n,q}(s^L_{n,j})w^L_{n,j}=\frac{2}{2p+1}\delta_{p,q}.\]
Due to the presence of the weakly singular term $ (t-s)^{\alpha-1} $ in the main problem \eqref{asli}, the weighted interpolatory quadrature formulae are
 utilized in the approximation procedure. For a function $ \phi \in \mathcal{P}_{{M_k}}(\Omega_k) $, the weighted quadrature formula is interpreted as \cite{wang2017} 
\begin{equation}\label{wtilda}
\int_{\Omega_k}(t-s)^{\alpha-1}\phi(s)\mathrm{d}s=\sum_{j=0}^{M_k} \tilde{w}^L_{k,j}(t)\phi(s^L_{k,j}), \quad t\in \Omega_n, \quad k<n,
\end{equation}
where $ \tilde{w}^L_{k,j}(t)=\int_{\Omega_k}(t-s)^{\alpha-1}l_{k,j}(s)\mathrm{d}s  $ 
and $ \big\{l_{k,j}(s)\big\}_{j=0}^{M_k} $ are Lagrange polynomials associated with the collocation points $ \{s^L_{k,j}\}_{j=0}^{M_k} $.
 The following lemma specifies the error bound of using the shifted Legendre-Gauss-Lobatto interpolation operator $ \mathcal{I}_{s,M_n}^L $ for each function which belongs  $H^{^m}(\Omega_n).$
\begin{lemma}(\cite{wang2017})\label{lemL}
For any ${v}\in H^{^m}(\Omega_n)$ with integer $1\leq m \leq M_n+1$, we get
	\[\Vert {v}-\mathcal{I}^L_{t,M_n} {v}\Vert_{\Omega_n}\leq c h_n^m(M_n+1)^{-m}\Vert \partial^m_t {v}\Vert_{\Omega_n}.\]
\end{lemma}
\subsection{The $hp$-collocation method for weakly singular integral equations}
\label{des of numerical scheme}
For a fixed integer $ N $, let $ \Omega_h:=\lbrace t_n:~0=t_0<t_1<\dots<t_N=T\rbrace$ be as a mesh on $\Omega$, $ h_n:=t_n-t_{n-1} $ and $h_{\max}=\max\limits_{1\leq n\leq N} h_n  $. Moreover, denote  $ u^{n}(t) $ as the solution of Eq. \eqref{asli} on the $ n$-th subinterval of $ \Omega,$ namely,
\[u^{n}(t)=u(t), \quad t\in \Omega_n:=(t_{n-1},t_n], \quad n=1, 2, \dots, N.\]
By the above mesh, we rewrite the Eq. \eqref{asli} as
\[\int_{0}^{t_{n-1}} (t-s)^{\alpha-1} \kappa(s,t) \psi(s,u(s))\mathrm{d}s+\int_{t_{n-1}}^{t} (t-s)^{\alpha-1} \kappa(s,t) \psi(s,u(s))\mathrm{d}s=f(t),\]
then for any $ t\in \Omega_n $, this equation can be written as
\begin{equation}\label{eq2}
\int_{t_{n-1}}^{t} (t-\tau)^{\alpha-1} \kappa(\tau,t) \psi(\tau,u^n(\tau))\mathrm{d}\tau=f(t)- \sum\limits_{k=1}^{n-1}\int_{\Omega_k} (t-s)^{\alpha-1}\kappa(s ,t) \psi(s ,u^{k}(s))\mathrm{d}s.
\end{equation} 
Now, we transfer the interval $(t_{n-1},t)$ to $ \Omega_n $ by the following linear transform
\begin{equation}\label{tau} \tau=\sigma(\lambda,t):=t_{n-1}+\dfrac{(\lambda-t_{n-1})(t-t_{n-1})}{h_n},\end{equation} 
to get
\begin{equation}\label{eq3}
\begin{aligned}
(\frac{t-t_{n-1}}{h_n})^\alpha\int_{\Omega_n}(t_n-\lambda)^{\alpha-1} \kappa\big(\sigma(\lambda,t),t\big) \psi\big(\sigma(\lambda,t),u^n(\sigma(\lambda,t))\big)\mathrm{d}\lambda
=&f(t)\\~-& \sum\limits_{k=1}^{n-1}\int_{\Omega_k} (t-s)^{\alpha-1}\kappa(s ,t) \psi(s ,u^{k}(s))\mathrm{d}s.
\end{aligned}
\end{equation} 
In the following, we mention some requirements considered in the next section. 
 Let $ \mathcal{I}^{\alpha-1,0}_{\lambda,{M_n}}:C(\Omega_n)\rightarrow \mathcal{P}_{M_n}(\Omega_n)$ be the Jacobi-Gauss interpolation operator. Now, we define a new Legendre-Gauss interpolation operator $ \mathcal{{I}}^{\alpha-1,0}_{\tau,M_n}:C(t_{n-1},t)\rightarrow \mathcal{P}_{M_n}(t_{n-1},t)$ owing to the relation \eqref{tau} with the following property
\[\mathcal{{I}}^{\alpha-1,0}_{\tau,M_n} g(\tau_{n,i})=g(\tau_{n,i}), \quad 0\leq i\leq M_n,\]
where $ \tau_{n,i}:=\tau_{n,i}(x)=\sigma(\lambda_{n,i},t) $ and $ \lambda_{n,i} $ are the $ M_n+1 $ Jacobi-Gauss quadrature nodes in $ \Omega_n $.
Clearly,
\[\mathcal{{I}}^{\alpha-1,0}_{\tau,M_n} g(\tau_{n,i})=g(\sigma(\lambda_{n,i},t))=\mathcal{I}^{\alpha-1,0}_{\lambda,{M_n}}g(\sigma(\lambda_{n,i},t)), \quad 0\leq i\leq M_n,\]
and by Eq. \eqref{jab}, we get
\begin{equation}\label{change}
\begin{split}
\int_{t_{n-1}}^{t}(t-\tau)^{\alpha-1}\mathcal{{I}}^{\alpha-1,0}_{\tau,M_n} g(\tau)\mathrm{d}\tau
&=(\frac{t-t_{n-1}}{h_n})^\alpha \int_{\Omega_n}(t_n-\lambda)^{\alpha-1}\mathcal{{I}}^{\alpha-1,0}_{\lambda,M_n}g(\sigma(\lambda,t))\mathrm{d}\lambda\\
&=(\frac{t-t_{n-1}}{2})^\alpha \sum_{j=0}^{M_n}g(\sigma(\lambda_{n,j},t))w_{n,j}\\
&=(\frac{t-t_{n-1}}{2})^\alpha \sum_{j=0}^{M_n}g(\tau_{n,j})w_{n,j}.
\end{split}
\end{equation}
Meanwhile, it is noticed that 
\begin{equation}\label{change2}
\int_{t_{n-1}}^{t}(t-\tau)^{\alpha-1}\big(\mathcal{{I}}^{\alpha-1,0}_{\tau,M_n} g(\tau)\big)^2\mathrm{d}\tau=(\frac{t-t_{n-1}}{2})^\alpha \sum_{j=0}^{M_n}g^2(\tau_{n,j})w_{n,j}.
\end{equation}
These equations will be valid for the Legendre interpolation operator $ \mathcal{I}_{M_n}^t $, if we take $ \alpha=1 $ and $ t=t_n. $ 
\subsubsection{The $ hp $-version of Jacobi-Gauss collocation method}
In order to seek the solution $ u^n_{M_n}(t)\in \mathcal{P}_{M_n}(\Omega_n)$ of Eq. (\ref{eq3}) by $hp$-collocation method, at the first step this equation is fully discretized as
\begin{equation}\label{eq6}
\begin{array}{ll}
\mathcal{I}^t_{M_n}\Big((\frac{t-t_{n-1}}{h_n})^\alpha\int_{\Omega_n}(t_n-\lambda)^{\alpha-1}\mathcal{I}_{\lambda,M_n}^{\alpha-1,0}\kappa\big(\sigma(\lambda,t),t\big) \psi(\sigma(\lambda,t),u_{M_n}^n(\sigma(\lambda,t)))\mathrm{d}\lambda
\Big)\\=\mathcal{I}^t_{M_n}(f(t))- \mathcal{I}^t_{M_n}\Big(\sum\limits_{k=1}^{n-1}\int_{\Omega_k} (t-s)^{\alpha-1}\mathcal{I}_{s,M_k}^{L}\kappa(s ,t) \psi(s ,u_{M_k}^{k}(s))\mathrm{d}s),\quad t\in \Omega_n  ,
\end{array}
\end{equation}
where
\begin{equation}\label{eq8}
\begin{split}
&\mathcal{I}^t_{M_n}u^n(t)=u^n_{M_n}(t)=\sum_{p=0}^{M_n}\hat{u}^{n}_p L_{n,p}(t), \\
&\mathcal{I}^t_{M_n}\mathcal{I}_{\lambda,M_n}^{\alpha-1,0}\Big((\frac{t-t_{n-1}}{h_n})^\alpha\kappa\big(\sigma(\lambda,t),t\big) \psi(\sigma(\lambda,t),u_{M_n}^n(\sigma(\lambda,t)))
\Big)=\sum_{p,q=0}^{M_n}a^{n}_{pq} L_{n,p}(t)J^{\alpha-1,0}_{n,q}(\lambda),\\
&\sum\limits_{k=1}^{n-1}\mathcal{I}^t_{M_n}\mathcal{I}_{s,M_k}^{L}\Big(\int_{\Omega_k} (t-s)^{\alpha-1}\kappa(s ,t) \psi(s ,u_{M_k}^{k}(s)))\mathrm{d}s=\sum\limits_{k=1}^{n-1}\mathcal{I}^t_{M_n}\Big( \sum\limits_{q=0}^{M_k} \tilde{w}^L_{k,q}(t)\kappa(t_{k,q}^{L},t) \psi(t_{k,q}^{L} ,u_{M_k}^{k}(t_{k,q}^{L})) \Big)\\&\hspace{3.15in}=\sum_{p=0}^{M_n}\sum\limits_{k=1}^{n-1}\sum_{q=0}^{M_k}b^k_{pq} L_{n,p}(t),
\end{split}
\end{equation}
and
\begin{equation}
\mathcal{I}^t_{M_n}f(t)=\sum_{p=0}^{M_n}\hat{f}^{n}_p L_{n,p}(t).
\end{equation}
Then, we get
\begin{equation}\label{eq9}
\begin{aligned}
\int_{\Omega_n}\frac{(t_n-\lambda)^{\alpha-1}}{h_n^\alpha}\mathcal{I}^t_{M_n}\mathcal{I}_{\lambda,M_n}^{\alpha-1,0}\Big((\frac{t-t_{n-1}}{h_n})^\alpha\kappa\big(\sigma(\lambda,t),t\big) \psi(\sigma(\lambda,t),u_{M_n}^n(\sigma(\lambda,t)))
\Big)\mathrm{d}\lambda\\
&\hspace{-0.9in}=\int_{\Omega_n}\frac{(t_n-\lambda)^{\alpha-1}}{h_n^\alpha}\sum_{p,q=0}^{M_n}a^{n}_{pq} L_{n,p}(t)J^{\alpha-1,0}_{n,q}(\lambda)\mathrm{d}\lambda \\&\hspace{-0.9in}=\sum_{p,q=0}^{M_n}a^{n}_{pq} L_{n,p}(t)\int_{\Omega_n}\frac{(t_n-\lambda)^{\alpha-1}}{h_n^\alpha}J^{\alpha-1,0}_{n,q}(\lambda)\mathrm{d}\lambda\\&\hspace{-0.9in}=\sum_{p=0}^{M_n}\hat{a}^{n}_{p0} L_{n,p}(t).
\end{aligned}
\end{equation}
 It is evident from Eqs. \eqref{eq8}-\eqref{eq9} that
\begin{equation}
\begin{aligned}
\hat{u}_p^n=&\frac{2p+1}{2}\sum_{i=0}^{M_n}u^n_{M_n}(t_{n,i})L_{n,p}(t_{n,i})w_{n,i},\\
\hat{a}_{p0}^n=&\frac{2p+1}{2^{1+\alpha}}\sum_{i,j=0}^{M_n}(t_{n,i}-t_{n-1})^{\alpha}\kappa(\sigma(t^{\alpha-1,0}_{n,j},t_{n,i}) ,t_{n,i}) \psi\big(\sigma(t^{\alpha-1,0}_{n,j},t_{n,i}) ,u^n_{M_n}(\sigma(t^{\alpha-1,0}_{n,j},t_{n,i}))\big)\\&L_{n,p}(t_{n,i})w_{n,i}w^{\alpha-1,0}_{n,j},\\
b_{pq}^k=&\frac{2p+1}{2}\sum_{i=0}^{M_n}
\tilde{w}^L_{k,q}(t_{n,i})\kappa(t_{k,q}^{L},t_{n,i}) \psi\big(t_{k,q}^{L} ,u_{M_k}^{k}(t_{k,q}^{L})\big)L_{n,p}(t_{n,i})w_{n,i},\\
\hat{f}_p^n=&\frac{2p+1}{2}\sum_{i=0}^{M_n}f^n_{M_n}(t_{n,i})L_{n,p}(t_{n,i})w_{n,i}.
\end{aligned}
\end{equation}
With Eqs. \eqref{eq8}-\eqref{eq9}, the equation \eqref{eq6} reads
\[\sum_{p=0}^{M_n}\hat{a}^{n}_{p0}L_{n,p}(t)=\sum_{p=0}^{M_n}\hat{f}^{n}_p L_{n,p}(t)+\sum_{p=0}^{M_n}\tilde{b}^{n}_{p} L_{n,p}(t),\]
where
\[\tilde{b}^{n}_p=\sum_{k=1}^{n-1}\sum_{q=0}^{M_k}b^{k}_{pq}.\]
Consequently, we compare the  coefficients to obtain 
\begin{equation}\label{eq7}
\hat{a}^{n}_{p0}=\hat{f}^{n}_p+\tilde{b}^{n}_{p},\quad 0\leq p \leq M_n.
\end{equation}
To evaluate  the unknown coefficients $ u^n_p $ for any given $ n $, we solve the nonlinear system (\ref{eq7}) with the Newton iteration method. Finally, the approximate solution can be obtained as
\begin{equation}\label{unm}
u_{_M}^{N}(t)=\sum_{n=1}^{N}\sum_{p=0}^{M_n}u^n_p L_{n,p}(t).
\end{equation}
It is worth to notice that for the linear case of Eq. \eqref{asli}, the unknown coefficients  $ \hat{u}^n_p $ for any given $ n $ can be obtained by the following linear system of equations
			\begin{equation}\label{eqr}
			A{\bf u}={\bf b}+{\bf c},
			\end{equation}
			where the entries of the matrix $ A=[a_{p,q}]_{p,q=0}^{M_n}$ are defined by
	\[
			a_{p,q}=\frac{2p+1}{2^{1+\alpha}}\sum_{i,j=0}^{M_n}(t_{n,i}-t_{n-1})^{\alpha}\kappa(\sigma(t^{\alpha-1,0}_{n,j},t_{n,i}) ,t_{n,i}) L_{n,q}\big(\sigma(t^{\alpha-1,0}_{n,j},t_{n,i})\big)L_{n,p}(t_{n,i})w_{n,i}w^{\alpha-1,0}_{n,j},
			\]
			and 
			\[{\bf u}=(\hat{u}^n_0, \dots ,\hat{u}^n_{M_n})^T,\quad {\bf b}=(\tilde{b}^n_0, \dots ,\tilde{b}^n_{M_n})^T,\quad {\bf c}=(\hat{f}^{n}_0, \dots,\hat{f}^{n}_{M_n})^T. \]
\section{Error analysis}\label{error analysis}
This section is devoted to the analysis of the introduced numerical scheme.  We shall characterize the $hp$-convergence of  the scheme \eqref{eq6} under the hypotheses of  Theorem \ref{thm1} and with respect to  Lemmas \ref{lem1}-\ref{lemL}. To this end, we first recall the  Gr\"{o}nwall inequality.  Throughout this paper, we denote $ \Vert . \Vert_D $ as $ L^2 $-norm on the interval $ D $ and $ M_{\min}=\min\limits_{1\leq n\leq N} M_n.$
\begin{lemma}(\cite{hack})\label{lem2} (Gr\"{o}nwall inequality)
	Assume that there are numbers $ \alpha,~ \beta_l\geq 0~ (l=0, 1, \dots, n-1)$ and $ 0 \leq M_{_0} <1 $ such that 
	\[ 0 \leq \varepsilon_n \leq \alpha+\sum_{l=0}^{n-1} \beta_l\varepsilon_l+M_{_0} \varepsilon_n, \quad n\geq 1.\]
	Then the quantities $ \varepsilon_n $ fulfill the following estimate for $ n\geq 0$
	\[\varepsilon_n\leq \frac{\alpha}{1-M_{_0}}\exp\big(\sum_{l=0}^{n-1}\frac{\beta_l}{1-M_{_0}}\big).\]
\end{lemma}
In what follows, some theoretical results regarding the convergence of the method are expressed. We notify that the main results regarding the error analysis of the proposed method are given by Theorem \ref{thm10} and Theorem \ref{thm9}. To provide some rigorous proofs for them, we need to define some auxiliary terms called  $B_1$, $B_2$ and $B_3$ as defined in Eq. \eqref{eq14}. 
	The next theorem provides an upper bound for the term $B_1$ which is the summation of the terms $B_2$ and $B_3$. 
\begin{theorem}\label{theorem11}
	Let ${u}^n$ be the solution of Eq. \eqref{eq3} under the hypothesis of Theorem \ref{thm1} and $u^n_{M_n}$ be the solution of Eq. \eqref{eq6}. According the assumptions of Theorem \ref{thm1}, the function $\psi(.,u) $ fulfills the Lipschitz condition with respect to the second variable, i.e.,
	\begin{equation}
	\vert \psi(.,u_1)-\psi(.,u_2)\vert\leq \gamma \vert u_1-u_2 \vert,\quad \gamma \geq 0.
	\end{equation}
	Then, for any $1\leq n\leq N$ and $ m\leq M_{\min}+1, $
	\[B_1=B_2+B_3,\] 
	with 
	\begin{equation}
	\begin{aligned}
	\Vert B_1\Vert^2_{\Omega_n}\leq &ch_n T^{2\alpha-1}\sum_{k=1}^{n-1} \Big(h_k^{2m}(M_k+1)^{-2m}\Vert \partial_s^m \psi(s ,u^k(s))\Vert_{\Omega_k}^2+\gamma^2 (\Vert e_k\Vert_{\Omega_k}^2+h_k^{2m-1}(M_k+1)^{-2m}\Vert \partial_t^m u\Vert_{\Omega_k}^2)\Big)\\
	&+ch_n^{2m} (M_n+1)^{-2m}\Vert \partial^m {f}\Vert_{\Omega_n}^2,
\end{aligned}
	\end{equation}
	where
	\begin{equation}\label{eq14}
	\begin{array}{ll}
	B_1=&\mathcal{I}^t_{M_n}\Big((\frac{t-t_{n-1}}{h_n})^\alpha\int_{\Omega_n}(t_n-\lambda)^{\alpha-1}\big(\mathcal{I}_{\lambda,M_n}^{\alpha-1,0}\big(\kappa\big(\sigma(\lambda,t),t\big) \psi(\sigma(\lambda,t),u_{M_n}^n(\sigma(\lambda,t)))\big)
	\\&-\kappa(\sigma(\lambda,t),t)\psi(\sigma(\lambda,t),u^n(\sigma(\lambda,t)))\big)\mathrm{d}\lambda\Big),\\
	B_2&={f}(t)-\mathcal{I}_{M_n}^t{f}(t),\\
	B_3&=\sum\limits_{k=1}^{n-1}\mathcal{I}^t_{M_n}\Big(\int_{\Omega_k} (t-s)^{\alpha-1}\big(\kappa(s ,t) \psi(s ,u^{k}(s))-\mathcal{I}_{s,M_k}^{L}\big(\kappa(s ,t) \psi(s ,u_{M_k}^{k}(s))\big)\big)\mathrm{d}s\Big),
	\end{array}
	\end{equation}
	and $ e_k={u}^{k}-u^{k}_{M_k}$ for $ 1\leq k\leq N$.
\end{theorem}
\begin{proof}
	Regarding Eq. \eqref{eq3}, we have
	\begin{equation}\label{eq12}
	\begin{aligned}
	\mathcal{I}_{M_n}^t\Big((\frac{t-t_{n-1}}{h_n})^\alpha\int_{\Omega_n}(t_n-\lambda)^{\alpha-1} \kappa\big(\sigma(\lambda,t),t\big) \psi(\sigma(\lambda,t),u^n(\sigma(\lambda,t)))\mathrm{d}\lambda\Big)
	=\mathcal{I}_{M_n}^t\big(f(t)\big)\\~-\mathcal{I}_{M_n}^t\Big( \sum\limits_{k=1}^{n-1}\int_{\Omega_k} (t-s)^{\alpha-1}\kappa(s ,t) \psi(s ,u^{k}(s))\mathrm{d}s\Big).
	\end{aligned}
	\end{equation} 
	By subtracting  (\ref{eq6}) from the above equation, we have
	\begin{equation}\label{b1}
	B_1(x)=B_2(x)+B_3(x),
	\end{equation}
	where the above terms are defined by \eqref{eq14}.
	In order to obtain an estimation for the term $B_1$, we need error bounds for $ \Vert B_i\Vert,~ i=2,3.$ First using Lemma \ref{lemf}, we infer that
	\begin{equation}
	\Vert B_2\Vert^2_{\Omega_n}= \Vert f(t)-\mathcal{I}_{M_n}^t{f}(t)\Vert_{\Omega_n}^2 
	\leq c h_n^{2m} (M_n+1)^{-2m}\Vert \partial^m {f}\Vert_{\Omega_n}^2.
	\end{equation}
	To seek an upper bound for $  \Vert B_3\Vert$, let us define,
	\[\begin{split}
	\kappa(s, t) \psi(s, u^{k}(s))-\mathcal{I}_{s,M_k}^{L}\big(\kappa(s ,t) \psi(s ,u_{M_k}^{k}(s))\big)=&\kappa(s ,t) \psi(s ,u^{k}(s))-\mathcal{I}_{s,M_k}^{L}\big(\kappa(s ,t) \psi(s, u^{k}(s))\big)\\&+\mathcal{I}_{s,M_k}^{L}\big(\kappa(s ,t) \psi(s ,u^{k}(s))\big)-\mathcal{I}_{s,M_k}^{L}\big(\kappa(s ,t) \psi(s ,u_{M_k}^{k}(s))\big)\\:&=\xi(s,t)+\eta(s,t). 
	\end{split}\]
	Therefore, using \eqref{change2} we get 
	\begin{equation}
	\begin{aligned}
	\Vert B_3\Vert^2&=\big\Vert\mathcal{I}^t_{M_n}\Big( \sum\limits_{k=1}^{n-1}\int_{\Omega_k} (t-s)^{\alpha-1}\big(\kappa(s ,t) \psi(s ,u^{k}(s))-\mathcal{I}_{s,M_k}^{L}\big(\kappa(s ,t) \psi(s ,u_{M_k}^{k}(s))\big)\big)\mathrm{d}s\Big)\big\Vert^2\\
	&=\Vert\mathcal{I}^t_{M_n}\Big( \sum\limits_{k=1}^{n-1}\int_{\Omega_k} (t-s)^{\alpha-1}\big(\xi(s,t)+\eta(s,t)\big)\mathrm{d}s\Big)\Vert^2\\
		&= \int_{\Omega_n}\Big[\mathcal{I}^t_{M_n}\Big( \sum\limits_{k=1}^{n-1}\int_{\Omega_k} (t-s)^{\alpha-1}\big(\xi(s,t)+\eta(s,t)\big)\mathrm{d}s\Big)\Big]^2\mathrm{d}t\\
	&=\frac{h_n}{2} \sum_{j=0}^{M_n} w_{n,j}\Big( \sum\limits_{k=1}^{n-1}\int_{\Omega_k} (t_{n,j}-s)^{\alpha-1}\big(\xi(s,t_{n,j})+\eta(s,t_{n,j})\big)\mathrm{d}s\Big)^2\\
	&\leq \frac{h_n}{2} \sum_{j=0}^{M_n} w_{n,j}\Big( \int_{0}^{t_{n-1}} (t_{n,j}-s)^{2\alpha-2}\mathrm{d}s\Big)\Big(\int_{0}^{t_{n-1}}\big(\xi(s,t_{n,j})+\eta(s,t_{n,j})\big)^2\mathrm{d}s\Big)\\
		&\leq c_\alpha h_n T^{2\alpha-1}\sum_{j=0}^{M_n} w_{n,j}\Big(\int_{0}^{t_{n-1}}\big(\xi(s,t_{n,j})+\eta(s,t_{n,j})\big)^2\mathrm{d}s\Big).
	\end{aligned}
	\end{equation}
	By virtue of the fact that $ \sum_{j=0}^{M_n} w_{n,j}=2$, we get
	\begin{equation}\label{b3b}
	\Vert B_3\Vert^2\leq ch_n T^{2\alpha-1}\int_{0}^{t_{n-1}}\big(\xi(s,t_{n,j})+\eta(s,t_{n,j})\big)^2\mathrm{d}s\leq ch_n T^{2\alpha-1} \Vert \xi(s,t_{n,j})+\eta(s,t_{n,j})\Vert_{L^2[0,t_{n-1}]}^2 .
	\end{equation}
	Minkowski inequality yields
	\begin{equation}\label{xiet}
	\Vert \xi+\eta\Vert_{L^2[0,t_{n-1}]}^2=\sum_{k=1}^{n-1} \Vert \xi_k+\eta_k\Vert_{\Omega_k}^2\leq 2\sum_{k=1}^{n-1} (\Vert \xi_k\Vert^2_{\Omega_k}+\Vert\eta_k\Vert_{\Omega_k}^2),
		\end{equation}
	where by Lemma \ref{lemL},
	 \begin{equation}\label{xik}
	 \begin{split}
	 \Vert \xi_k\Vert^2 &\leq  \Vert (\mathcal{I}-\mathcal{I}_{s,M_k}^{L})\big(\kappa(s ,t_{n,j}) \psi(s ,u^{k}(s))\big)\Vert^2\\&\leq ch_k^{2m}(M_k+1)^{-2m} \Vert \partial_s^m\kappa(s ,t_{n,j}) \psi(s ,u^{k}(s))\Vert^2\\&\leq ch_k^{2m}(M_k+1)^{-2m}\Vert \partial_s^m \psi(s ,u(s))\Vert_{\Omega_k}^2,
	  \end{split}
	 \end{equation}
	in which the last inequality follows from the assumption that $ \kappa \in C^m(\Omega\times \Omega). $ Moreover, it is deduced that 
	\begin{equation}\label{etak}
		 \begin{split}
		 \Vert \eta_k\Vert^2 &\leq  \Vert \mathcal{I}_{s,M_k}^{L}\Big(\kappa(s ,t_{n,j}) \big(\psi(s ,u^{k}(s))- \psi(s, u_{M_k}^{k}(s)) \big)\Big)\Vert^2\\&=\frac{h_k}{2}  \sum_{j=0}^{M_k}\Big[ \kappa(s^L_{k,j} ,t_{n,j}) \big(\psi(s^L_{k,j} , u^{k}(s^L_{k,j}))- \psi(s^L_{k,j}, u_{M_k}^{k}(s^L_{k,j}))\big)\Big]^2w_{k,j}
		\\
		&\leq c  \gamma^2 h_k \sum_{j=0}^{M_k} \big(u^{k}(s^L_{k,j})-u_{M_k}^{k}(s^L_{k,j})\big)^2w_{k,j}
	\\
	& \leq c \gamma^2 \int_{\Omega_k} \big(\mathcal{I}^L_{\tau,M_k}\big(u^{k}(\tau)-u_{M_k}^{k}(\tau)\big)^2 \mathrm{d}\tau
		\\
		& \leq c \gamma^2 \int_{\Omega_k} \big(\mathcal{I}^L_{\tau,M_k}u^{k}(\tau)-u^{k}(\tau)\big)^2+\big(u^{k}(\tau)-u_{M_k}^{k}(\tau))\big)^2 \mathrm{d}\tau\\
		&\leq c\gamma^2 (\Vert e_k\Vert^2_{\Omega_k}+h_k^{2m}(M_k+1)^{-2m}\Vert \partial_t^m u\Vert_{\Omega_k}^2).
		 	 \end{split}
		 \end{equation}
	Now, utilizing \eqref{xiet}-\eqref{etak}, the term \eqref{b3b} can be simplified as 
	\begin{equation}
	\begin{split}
	\Vert B_3\Vert_{\Omega_n}^2&\leq ch_n T^{2\alpha-1} \sum_{k=1}^{n-1} \big(\Vert \xi_k\Vert_{\Omega_k}^2+\Vert \eta_k\Vert^2_{\Omega_k}\big)\\& \leq ch_n T^{2\alpha-1}\sum_{k=1}^{n-1} \Big(h_k^{2m}(M_k+1)^{-2m}\Vert \partial_s^m \psi(s ,u(s))\Vert_{\Omega_k}^2+\gamma^2 (\Vert e_k\Vert_{\Omega_k}^2+h_k^{2m}(M_k+1)^{-2m}\Vert \partial_t^m u\Vert_{\Omega_k}^2)\Big),
		\end{split}
	\end{equation}
		so, the desired result is deduced from $ \Vert B_1\Vert^2\leq 2(\Vert B_2\Vert^2+ \Vert B_3\Vert^2).$
\end{proof}
The next lemma investigate an upper bound for the  auxiliary term  $B_0$,  introduced in Eq. \eqref{b0d}. 
\begin{lemma}\label{lemb}
	Under the hypotheses of the previous theorem, the term
	\begin{align}\label{b0d}\begin{split}	B_0(t)=&\mathcal{I}^t_{M_n}\Big((\frac{t-t_{n-1}}{h_n})^\alpha\int_{\Omega_n}(t_n-\lambda)^{\alpha-1}\big(\mathcal{I}_{\lambda,M_n}^{\alpha-1,0}\big(\kappa\big(\sigma(\lambda,t),t\big) \psi(\sigma(\lambda,t),u_{M_n}^n(\sigma(\lambda,t)))\big)\mathrm{d}\lambda\Big)\\
	&-(\frac{t-t_{n-1}}{h_n})^\alpha\int_{\Omega_n}(t_n-\lambda)^{\alpha-1}\kappa(\sigma(\lambda,t),t)\psi(\sigma(\lambda,t),u^n(\sigma(\lambda,t)))\big)\mathrm{d}\lambda,
	\end{split}\end{align}
	has the following error bound
	\begin{equation}\label{b0}\begin{split}	\Vert B_0\Vert^2\leq &ch_n T^{2\alpha-1}\sum_{k=1}^{n-1} \Big(h_k^{2m}(M_k+1)^{-2m}\Vert \partial_s^m \psi(s ,u(s))\Vert_{\Omega_k}^2+\gamma^2 (\Vert e_k\Vert_{\Omega_k}^2+h_k^{2m}(M_k+1)^{-2m}\Vert \partial_t^m u\Vert^2)\Big)\\
	&+ch_n^{2m} (M_n+1)^{-2m}\Vert \partial^m {f}\Vert_{\Omega_n}^2+c_\alpha \,h_n^{2m+\alpha+1}\,(M_n+1)^{-2m} \Vert \psi(.,u(.))\Vert^2_{
{\chi^{\alpha-1,0}}_{(\Omega_n)}	
}.
	\end{split}
	\end{equation}  
\end{lemma}
\begin{proof}
	By adding and subtracting the term \[\mathcal{I}_{M_n}^t\Big((\frac{t-t_{n-1}}{h_n})^\alpha\int_{\Omega_n}(t_n-\lambda)^{\alpha-1} \kappa\big(\sigma(\lambda,t),t\big) \psi(\sigma(\lambda,t),u^n(\sigma(\lambda,t)))\mathrm{d}\lambda\Big),\]
	to the term $ B_0(t),$ we have
		\[	 B_0(t)= B_1(t)+B_4(t),\]
		where $ B_1(t)$ is defined by \eqref{eq14} and 
		\[B_4(t)=(\mathcal{I}_{M_n}^t-\mathcal{I})\Big((\frac{t-t_{n-1}}{h_n})^\alpha\int_{\Omega_n}(t_n-\lambda)^{\alpha-1}\kappa(\sigma(\lambda,t),t)\psi(\sigma(\lambda,t),u^n(\sigma(\lambda,t)))\Big)\mathrm{d}\lambda.\] 
		In order to obtain an estimation for $ B_0(t)$, it suffices to seek an upper bound for $ B_4$. Therefore, using operator norm and Lemma \ref{lemf},
	\begin{equation}\label{b4}
		\begin{split}
		\Vert B_4 \Vert^2_{\Omega_n} &\leq \Vert \mathcal{I}_{M_n}^t -\mathcal{I} \Vert^2 \Vert \int_{t_{n-1}}^{t} (t-\tau)^{\alpha-1}\kappa(\tau,t)\psi(\tau,u^n(\tau)) \mathrm{d}\tau\Vert^2\\
		&\leq c \,h_n^{2m}\,(M_n+1)^{-2m}\int_{\Omega_n}\Big(\int_{t_{n-1}}^{t} (t-\tau)^{\alpha-1}\kappa(\tau,t)\psi(\tau,u^n(\tau)) \mathrm{d}\tau\Big)^2\mathrm{d}t\\
		&\leq c \,h_n^{2m}\,(M_n+1)^{-2m}\int_{\Omega_n}\big(\int_{t_{n-1}}^{t} (t-\tau)^{\alpha-1}\mathrm{d}\tau\big)\big(\int_{t_{n-1}}^{t} (t-\tau)^{\alpha-1}\big(\kappa(\tau,t)\psi(\tau,u^n(\tau))\big)^2 \mathrm{d}\tau\big)\mathrm{d}t\\
		& \leq c_\alpha \,h_n^{2m+\alpha+1}\,(M_n+1)^{-2m} \Vert \psi(.,u^n(.))\Vert^2_{L^2_{\chi^{\alpha-1,0}}(\Omega_n)}.
			\end{split}
		\end{equation}
\end{proof} 

 The main theorem concerned with the $hp$-convergence analysis of the scheme \eqref{eq6} for each subinterval $\Omega_n$ is provided as follows:  
\begin{theorem}\label{thm10}\label{te1}
	Assume that the Fr{\'e}chet derivative of the operator $ \mathcal{K}u $ with respect to $ u $ is satisfied at
$\vert (\mathcal{K}^\prime u)(t) \vert \geq l > 0,$ 
	then under the hypothesis of the Theorem \ref{theorem11}, for sufficiently small $ h_{\max} $ the following error estimate is obtained
		\begin{equation}
		\begin{aligned}
		\Vert e_n \Vert^2 = \Vert  {u}^n -{u}^n_{M_n}\Vert^2 \leq&\frac{c_\alpha}{\delta^2}\exp(c \gamma^2 T^{2\alpha})\Big( T^{2\alpha-1}\sum_{k=1}^{n-1} \Big(h_k^{2m}(M_k+1)^{-2m}\Vert \partial_s^m \psi(s ,u(s))\Vert_{\Omega_k}^2\\
		&+\gamma^2 h_k^{2m}(M_k+1)^{-2m}\Vert \partial_t^m u\Vert_{\Omega_k}^2\Big)+h_n^{2m-1} (M_n+1)^{-2m}\big(\Vert \partial^m {f}\Vert_{\Omega_n}^2\\
		&+ h_n^{2\alpha} \Vert \psi(.,u_{M}^N(.))\Vert^2_{H^m_{\chi^{\alpha-1,0}}({\Omega_n})}\big) + h_{n}^{2m+\alpha}(M_n+1)^{-2m}\big(\gamma^2\Vert u\Vert^2_{H^m_{\chi^{\alpha-1,0}}(\Omega_n)}\\
		&+\Vert\psi(.,u(.))\Vert^2_{H^m_{\chi^{\alpha-1,0}}(\Omega_n)}\big)\Big).
		\end{aligned}
		\end{equation}
\end{theorem}
\begin{proof}
	For convenience, let
	\begin{equation}
	F(t,\tau,u(\tau)):= (t-\tau)^{\alpha-1}{\kappa}(\tau,t) \psi(\tau,u(\tau)),\hspace{0.2in}\tau \in (t_{n-1},t].
	\end{equation} and $\mathcal{G}u(t):=\int_{t_{n-1}}^{t}F(t,\tau,u(\tau))) \mathrm{d}\tau.$
	Under the mean value theorem \cite[p. 229]{atkinson}, we have
	\begin{equation}
	\int_{t_{n-1}}^{t}F(t,\tau,{u}^{n}(\tau)) \mathrm{d}\tau  - \int_{t_{n-1}}^{t}F(t,\tau,{u}_{M_n}^{n}(\tau))\mathrm{d}\tau= \mathcal{G}^\prime (\xi)\big({u}^{n}(t) - {u}_{M_n}^{n}(t)\big),
	\end{equation}
	where $ \xi\in (\min\lbrace {u}^{n},{u}_{M_n}^{n}\rbrace,\max\lbrace {u}^{n},{u}_{M_n}^{n}\rbrace ) $ and $ \mathcal{G}^\prime $ denotes the Fr{\'e}chet derivative of $ \mathcal{G} $, namely, 
	\[ \mathcal{G}^\prime (u) h(t)=\int_{t_{n-1}}^{t} (t-\tau)^{\alpha-1}{\kappa}(\tau,t) \frac{\partial \psi(\tau,u(\tau)) }{\partial u} h(\tau) \mathrm{d}\tau. \]
	It is well-known that 
		\[ \mathcal{K}^\prime (u) h(t)=\int_{0}^{t} (t-s)^{\alpha-1}{\kappa}(s,t) \frac{\partial \psi(s,u(s)) }{\partial u} h(s) \mathrm{d}s. \]
	Since  $\vert (\mathcal{K}^\prime u) (t)\vert \gg 0$, then one can deduce that  $\delta:=\vert \mathcal{G}^\prime (u) h(t)\vert \gg 0.$
	Therefore,
	\begin{equation}\label{eq15}
	\Big \vert{u}^{n}(t) - {u}_{M_n}^{n}(t) \Big \vert \leq \frac{1}{\delta} \Big \vert\int_{t_{n-1}}^{t}F(t,\tau,{u}^{n}(\tau) \mathrm{d}\tau  - \int_{t_{n-1}}^{t}F(t,\tau,{u}_{M_n}^{n}(\tau))\mathrm{d}\tau\Big|.
	\end{equation}
	It is evident that 
	\begin{equation}
	\int_{t_{n-1}}^{t}F(t,\tau  ,{u}^{n}(\tau)) \mathrm{d}\tau =(\frac{t-t_{n-1}}{h_n})^\alpha\int_{\Omega_n}(t_n-\lambda)^{\alpha-1}\kappa(\sigma(\lambda,t),t)\psi\big(\sigma(\lambda,t),u^n(\sigma(\lambda,t))\big)\mathrm{d}\lambda,
	\end{equation}
	so from (\ref{eq15}), we infer that
	\begin{equation}\label{en0}
	\begin{aligned}
	\Big \vert e_n(t)\Big \vert=\Big \vert {u}^{n}(t) - {u}_{M_n}^{n}(t) \Big \vert &\leq \frac{1}{\delta} \Big \vert (\frac{t-t_{n-1}}{h_n})^\alpha\int_{\Omega_n}(t_n-\lambda)^{\alpha-1}\kappa(\sigma(\lambda,t),t)\Big[\psi\big(\sigma(\lambda,t),u^n(\sigma(\lambda,t))\big)\\&\,\,\,\, -\psi\big(\sigma(\lambda,t),u_{M_n}^n(\sigma(\lambda,t))\big)\Big]\mathrm{d}\lambda \Big \vert\\&	
	\leq \frac{1}{\delta}\big(\vert B_{0}(t)\vert + E_{1}(t) + E_{2}(t)\big),
	\end{aligned}
	\end{equation}
	where $ B_0(t) $ is defined by Lemma \ref{lemb}  and
	\begin{equation}
	\begin{aligned}
	E_1(t)&=\Big \vert (\mathcal{I}_{M_n}^{x}-\mathcal{I})\Big((\frac{t-t_{n-1}}{h_n})^\alpha\int_{\Omega_n}(t_n-\lambda)^{\alpha-1}\mathcal{I}_{\lambda,M_n}^{\alpha-1,0}\big(\kappa(\sigma(\lambda,t),t)\psi\big(\sigma(\lambda,t),u_{M_n}^n(\sigma(\lambda,t))\big)\mathrm{d}\lambda\Big)\Big\vert, \\
	E_2(t)&=\Big \vert (\frac{t-t_{n-1}}{h_n})^\alpha\int_{\Omega_n}(t_n-\lambda)^{\alpha-1}(\mathcal{I}_{\lambda,M_n}^{\alpha-1,0}-\mathcal{I})\Big(\kappa(\sigma(\lambda,t),t)\psi\big(\sigma(\lambda,t),u_{M_n}^n(\sigma(\lambda,t))\big)\Big)\mathrm{d}\lambda\Big \vert,
	\end{aligned}
	\end{equation}
	thus 
	\begin{equation}\label{en1}	\Vert e_n\Vert_{\Omega_n}^2\leq \frac{3}{\delta^2}( \Vert B_0 \Vert_{\Omega_n}^2+ \Vert E_1 \Vert_{\Omega_n}^2+ \Vert E_2\Vert_{\Omega_n}^2).
	\end{equation}
The error estimations for the terms $ \Vert E_i\Vert_{\Omega_n},$  with  $i=1,2, $ as calculated in Appendix A  are
	\begin{equation}
	\begin{aligned}
	\Vert E_{1}\Vert^2\leq c_\alpha h_n^{\alpha} \Big(h_n^{\alpha} \gamma^2 \Vert e_n\Vert^2_{\Omega_n} +h_{n}^{2m+1}(M_n+1)^{-2m}\big(\gamma^2\Vert u^n\Vert^2_{H^m_{\chi^{\alpha-1,0}}(\Omega_n)}+\Vert\psi(.,u^n(.))\Vert^2_{H^m_{\chi^{\alpha-1,0}}(\Omega_n)}\big)\Big),
	\end{aligned}
	\end{equation}
	and
	\begin{equation}
	\Vert E_2 \Vert_{\Omega_n}^2 \leq c_\alpha h_n^{2m+2\alpha}(M_n+1)^{-2m} \Vert \psi(.,u_{M_n}^n(.))\Vert^2_{H^m_{\chi^{\alpha-1,0}}(\Omega_n)}.
	\end{equation}
		Hence, the desired result follows  from the relation \eqref{en1} 
							\begin{equation}\label{en}\begin{split}	
			\Vert e_n\Vert_{\Omega_n}^2\leq &\frac{c_\alpha}{\delta^2}\Big(h_n T^{2\alpha-1}\sum_{k=1}^{n-1} \Big(h_k^{2m}(M_k+1)^{-2m}\Vert \partial_s^m \psi(s ,u(s))\Vert_{\Omega_k}^2+\gamma^2 (\Vert e_k\Vert_{\Omega_k}^2+h_k^{2m}(M_k+1)^{-2m}\Vert \partial_t^m u\Vert_{\Omega_k}^2)\Big)\\
				&+h_n^{2m} (M_n+1)^{-2m}\big(\Vert \partial^m {f}\Vert_{\Omega_n}^2+ h_n^{2\alpha} \Vert \psi(.,u_{M}^N(.))\Vert^2_{H^m_{\chi^{\alpha-1,0}}({\Omega_n})}\big)\\
							& +h_n^{\alpha} \Big(h_n^{\alpha} \gamma^2 \Vert e_n\Vert^2_{\Omega_n} +h_{n}^{2m+1}(M_n+1)^{-2m}\big(\gamma^2\Vert u^n\Vert^2_{H^m_{\chi^{\alpha-1,0}}(\Omega_n)}+\Vert\psi(.,u^n(.))\Vert^2_{H^m_{\chi^{\alpha-1,0}}(\Omega_n)}\big)\Big),
					\end{split}
		\end{equation} 
	or equivalently,  
	\begin{equation}\label{en}\begin{split}	
	(1-\frac{c_\alpha\gamma^2h_n^{2\alpha}}{\delta^2})	\Vert e_n\Vert^2\leq &\frac{c_\alpha}{\delta^2}\Big(h_n T^{2\alpha-1}\sum_{k=1}^{n-1} \Big(h_k^{2m}(M_k+1)^{-2m}\Vert \partial_s^m \psi(s ,u(s))\Vert_{\Omega_k}^2+\gamma^2 (\Vert e_k\Vert_{\Omega_k}^2\\
				&+h_k^{2m}(M_k+1)^{-2m}\Vert \partial_t^m u\Vert_{\Omega_k}^2)\Big)+h_n^{2m} (M_n+1)^{-2m}\big(\Vert \partial^m {f}\Vert_{\Omega_n}^2+ h_n^{2\alpha} \Vert \psi(.,u_{M}^N(.))\Vert^2_{H^m_{\chi^{\alpha-1,0}}({\Omega_n})}\big)\\
																& +h_{n}^{2m+\alpha+1}(M_n+1)^{-2m}\big(\gamma^2\Vert u\Vert^2_{H^m_{\chi^{\alpha-1,0}}(\Omega_n)}+\Vert\psi(.,u(.))\Vert^2_{H^m_{\chi^{\alpha-1,0}}(\Omega_n)}\big)\Big).
					\end{split}
				\end{equation}  
Next, assume that the step size  $ h_{\max} $ is sufficiently small such that
	\[ d \gamma^2 h^{2\alpha}_{\max}\leq \beta <1.\]			
Therefore, by means of Gr\"{o}nwall inequality  and taking $ \varepsilon_k=h_k^{-1}\Vert e_k\Vert^2 $, we have
	\begin{equation}
	\begin{aligned}
	\Vert e_n \Vert_{\Omega_n}^2 \leq  &\frac{c_\alpha}{\delta^2}\exp(c \gamma^2 T^{2\alpha})\Big( T^{2\alpha-1}\sum_{k=1}^{n-1} \Big(h_k^{2m}(M_k+1)^{-2m}\Vert \partial_s^m \psi(s ,u(s))\Vert_{\Omega_k}^2+\gamma^2 h_k^{2m}(M_k+1)^{-2m}\Vert \partial_t^m u\Vert_{\Omega_k}^2\Big)\\
	&+h_n^{2m-1} (M_n+1)^{-2m}\big(\Vert \partial^m {f}\Vert_{\Omega_n}^2+
	h_n^{2\alpha} \Vert \psi(.,u_{M}^N(.))\Vert^2_{H^m_{\chi^{\alpha-1,0}}({\Omega_n})}\big)\\
	& + h_{n}^{2m+\alpha}(M_n+1)^{-2m}\big(\gamma^2\Vert u\Vert^2_{H^m_{\chi^{\alpha-1,0}}(\Omega_n)}+\Vert\psi(.,u(.))\Vert^2_{H^m_{\chi^{\alpha-1,0}}(\Omega_n)}\big)\Big),
	\end{aligned}
	\end{equation}
	which infers the desired result.
\end{proof}
 The final convergence rate of the $hp$-method at the  whole interval $\Omega$ can be directly obtained by the above theorem. 
\begin{theorem}\label{thm9}
Assume that $ u(t) $ be the exact solution of Eq. (\ref{asli}) and $ u_{_M}^N(t) $ be the global approximate solution obtained from Eq. (\ref{unm}). Under the hypothesis of Theorem \ref{te1}, the following error estimate can be derived for sufficiently small $ h_{\max} $ as
\begin{equation}
\begin{aligned}
\Vert u-u_{_M}^N\Vert_{\Omega}\leq& \frac{c_\alpha}{\delta}\exp(c \gamma^2 T^{2\alpha})h_{\max}^{m}(M_{\min}+1)^{-m}\Big(T^{\alpha}(\gamma\Vert \partial_t^m u\Vert_{\Omega}+\Vert \partial_s^m \psi(s ,u(s))\Vert_{\Omega})+\Vert \partial^m {f}\Vert_{\Omega}\\
&+ \gamma\Vert u\Vert_{H^m_{\chi^{\alpha-1,0}}(\Omega)}+ \Vert\psi(.,u(.))\Vert_{H^m_{\chi^{\alpha-1,0}}(\Omega)}+ h_{\max}^{\alpha}\Vert \psi(.,u_{M}^N(.))\Vert_{H^m_{\chi^{\alpha-1,0}}({\Omega})}\Big).
\end{aligned}
\end{equation}
\end{theorem}
\begin{proof}
The global convergence error of the approximate solution $ u_{_M}^N(t) $ which is given by 
\[ u_{_M}^N(t)\rvert_{ \Omega_n}={u}^n_{M_n}(x)\Big\rvert_{x=\frac{2t-t_{n-1}-t_n}{h_n}}, \quad 1\leq n\leq N, \]
and the exact solution  $ u(t) $ which is fulfilled in
\[ u(t)\rvert_{\Omega_n}={u}^n(x)\Big\rvert_{x=\frac{2t-t_{n-1}-t_n}{h_n}}, \quad 1\leq n\leq N,\]
can be easily obtained using Theorem \ref{te1}  and the following formula
\[\Vert u-u_{_M}^N\Vert^2_{\Omega}=\frac{1}{2}\sum_{n=1}^{N}h_n\Vert e_n\Vert_{\Omega_n}^2.\]
Therefore,
	\begin{equation}\label{uuu}
	\begin{aligned}
	\Vert   u-u_{_M}^N\Vert^2 \leq&\frac{c_\alpha}{\delta^2}\exp(c \gamma^2 T^{2\alpha})\sum_{n=1}^{N}\Big(h_n T^{2\alpha-1}\sum_{k=1}^{n-1} \Big(h_k^{2m}(M_k+1)^{-2m}\Vert \partial_s^m \psi(s ,u(s))\Vert_{\Omega_k}^2+\gamma^2 h_k^{2m}(M_k+1)^{-2m}\Vert \partial_t^m u\Vert_{\Omega_k}^2\Big)\\
	&+h_n^{2m} (M_n+1)^{-2m}\big(\Vert \partial^m {f}\Vert_{\Omega_n}^2+ h_n^{2\alpha+1} \Vert \psi(.,u_{M}^N(.))\Vert^2_{H^m_{\chi^{\alpha-1,0}}({\Omega_n})}\big)\\
	& +h_{n}^{2m+\alpha+1}(M_n+1)^{-2m}\big(\gamma^2\Vert u\Vert^2_{H^m_{\chi^{\alpha-1,0}}(\Omega_n)}+\Vert\psi(.,u(.))\Vert^2_{H^m_{\chi^{\alpha-1,0}}(\Omega_n)}\big)\Big).
	\end{aligned}
	\end{equation}
All terms of the above error bound can be simplified  using $ h_{\max} $ and $ M_{\min} $ as follows
\[\sum_{n=1}^{N}h_n^{2m} (M_n+1)^{-2m}\Vert \partial^m {f}\Vert^2_{\Omega_n}\leq h_{\max}^{2m} (M_{\min}+1)^{-2m}\Vert \partial^m {f}\Vert^2_{\Omega},\]
similarly,
 \[\sum_{n=1}^{N}h_n^{2m+2\alpha+1} (M_n+1)^{-2m} \Vert \psi(.,u_{M}^N(.))\Vert^2_{H^m_{\chi^{\alpha-1,0}}({\Omega_n})} \leq h_{\max}^{2m+2\alpha+1} (M_{\min}+1)^{-2m}  \Vert \psi(.,u_{M}^N(.))\Vert^2_{H^m_{\chi^{\alpha-1,0}}({\Omega})}.\]
Also the following inequalities can be proved 
\[\sum_{n=1}^{N}h_n^{2m+\alpha+1}(M_n+1)^{-2m} \gamma^2\Vert u\Vert^2_{H^m_{\chi^{\alpha-1,0}}(\Omega_n)}\leq
h_{\max}^{2m+\alpha+1} (M_{\min}+1)^{-2m} \gamma^2\Vert u\Vert^2_{H^m_{\chi^{\alpha-1,0}}(\Omega)},\]
and
\[\sum_{n=1}^{N}h_n^{2m+\alpha+1}(M_n+1)^{-2m} \Vert\psi(.,u(.))\Vert^2_{H^m_{\chi^{\alpha-1,0}}(\Omega_n)}\leq
h_{\max}^{2m+\alpha+1} (M_{\min}+1)^{-2m} \Vert\psi(.,u(.))\Vert^2_{H^m_{\chi^{\alpha-1,0}}(\Omega)}.\]
Furthermore, we can obtain
\[\begin{aligned}\sum_{n=1}^{N}h_n T^{2\alpha-1}\sum_{k=1}^{n-1} \gamma^2 h_k^{2m}(M_k+1)^{-2m}\Vert \partial_t^m u\Vert_{\Omega_k}^2& \leq \gamma^2h_{\max}^{2m}(M_{\min}+1)^{-2m}T^{2\alpha-1}\sum_{n=1}^{N}h_n \sum_{k=1}^{n-1}\Vert \partial_t^m u\Vert_{\Omega_k}^2\\&\leq \gamma^2 h_{\max}^{2m}(M_{\min}+1)^{-2m}T^{2\alpha}\Vert \partial_t^m u\Vert_{\Omega}^2,	\end{aligned}\]
and
\[\begin{aligned}\sum_{n=1}^{N}h_n T^{2\alpha-1}\sum_{k=1}^{n-1} h_k^{2m}(M_k+1)^{-2m}\Vert \partial_s^m \psi(s ,u(s))\Vert_{\Omega_k}^2& \leq h_{\max}^{2m}(M_{\min}+1)^{-2m}T^{2\alpha-1}\sum_{n=1}^{N}h_n \sum_{k=1}^{n-1}\Vert \partial_s^m \psi(s ,u(s))\Vert_{\Omega_k}^2\\&\leq h_{\max}^{2m}(M_{\min}+1)^{-2m}T^{2\alpha} \Vert \partial_s^m \psi(s ,u(s))\Vert_{\Omega}^2.	\end{aligned}.\]
Consequently, the combination of the above error bounds for Eq. (\ref{uuu}) leads to  
\begin{equation}
\begin{aligned}
\Vert u-u_{_M}^N\Vert^2_{\Omega}\leq& \frac{c_\alpha}{\delta^2}\exp(c \gamma^2 T^{2\alpha})h_{\max}^{2m}(M_{\min}+1)^{-2m}\Big(T^{2\alpha}(\Vert \partial_t^m u\Vert_{\Omega}^2+\Vert \partial_s^m \psi(s ,u(s))\Vert_{\Omega}^2)+{\Vert \partial^m {f}\Vert^2_{\Omega}}\\
&+ h_{\max}^{\alpha+1}\big( \gamma^2\Vert u\Vert^2_{H^m_{\chi^{\alpha-1,0}}(\Omega)}+ \Vert\psi(.,u(.))\Vert^2_{H^m_{\chi^{\alpha-1,0}}(\Omega)}+ h_{\max}^{\alpha}\Vert \psi(.,u_{M}^N(.))\Vert^2_{H^m_{\chi^{\alpha-1,0}}({\Omega})}\big)\Big),\\
\end{aligned}
\end{equation}
which completes the proof.
\end{proof}

\section{Numerical results}\label{numerical experiments}
This section illustrates some numerical experiments in order to  scrutinize the efficiency of the $hp$-collocation method for the Abel integral equations. 
The experiments are implemented in $\textsl{Mathematica}^{\circledR}$ software platform and the programs are executed on a PC with 3.50 GHz Intel(R) Core(TM) i5-4690K processor.
In order to analyze the method, the following notations are introduced:
\[E_{1}(u_{M}^{N})=\left(\sum\limits_{k=1}^{N}\sum\limits_{j=0}^{M_k}\frac{h_k}{2}w_{k,j}\big(u^k(x_{k,j})-u_{M_k}^k(x_{k,j})\big)^2 \right)^{\frac{1}{2}},\]
\[E_{2}(u_{M}^{N})=\max\limits_{t\in \Omega}\big\vert u(t)-u_{_M}^{N}(t))\big\vert.\]
{The discrete $ L^{2}$-norm error is denoted by $E_{1}(u_{M}^{N})$, 
 and also $E_{2}(u_{M}^{N}) $ indicates  the infinite norm. Furthermore, the order of convergence $\rho_{N}$ is defined by $\log_{2}\big(\dfrac{E_{1}(u_{M}^{N})}{E_{1}(u_{M}^{2N})}\big) $. The relation of the theoretical order of convergence stated in Theorem \ref{thm9} and  $\rho_{N}$ is derived as 
	\begin{equation}
	\label{ro}
	\rho_N=\log_{2}\big(\dfrac{E_{1}(u_{M}^{N})}{E_{1}(u_{M}^{2N})}\big) \approx \log_{2}\big(\dfrac{ch_{\max}^mM_{\min}^{-m}}{c(\dfrac{h_{\max}}{2})^mM_{\min}^{-m}}\big)=\log_2 2^m=m.
	\end{equation}
This criterion  can be utilized to check the order  of convergence in practice based on the continuous injection between $L^2(\Omega)$ and $L^{\infty}(\Omega)$ \cite{brezis}.

	Let $L$ denote the number of unknown coefficients; in this way we have $L=\sum\limits_{n=1}^N(M_n+1)$ for the $hp$-collocation method and in a specific case, if all degrees of polynomials $  M_n $ are equal, i.e. $ M_n=M^*$, for $ n=1, \dots, N,$ then according to relation \eqref{unm}, $ L=(M^*+1)\times N. $ For convenience, we denote $ M:=M^*+1,$  so $ L=M\times N. $}

The nonlinear systems which arise in the formulation of the method are solved by utilizing  the Newton iteration method which needs an initial guess.
In these examples, all the initial points are chosen by an algorithm based on the steepest descent method.

\begin{remark}\label{rem3}
		In \cite{nedai, patterson}, two adaptive schemes based on mesh refinement are introduced. A bound for the error is chosen arbitrary and then 
		the implementation proceeds  by increasing the degree of polynomials or refinement of the mesh size until the desired error bound is observed.  The described scheme is called ``adaptive $hp$-collocation method". 
\end{remark}
\subsubsection*{Singular solution}
\begin{example}\label{ex111}
Consider a test problem with singular solution
	\begin{equation*}
	\int_{0}^{t} (t-s)^{\alpha-1}\exp(ts) u^2(s)\mathrm{d}s=(\frac{1}{t})^{-2\alpha} t^{2+\alpha}
	\Gamma(3+2\alpha)\Gamma(\alpha)~ _1F_1(3+2\alpha,
	3+3\alpha, t^2), \quad t\in [0,1],
	\end{equation*}
where the function $_1F_1$ is called confluent hypergeometric function of the first kind. The exact solution $ u(t)=t^{1+\alpha}$  belongs to $ H^2_{\alpha-1,0}([0,1])$.  
In this example, different merits of the $hp$- method are investigated. First and foremost, the superiority of the $hp$- version method against $h$- and $p$-version method with $\alpha=0.3$ is demonstrated by Figure \ref{f1} and \ref{f11}. The $hp$-version method allows us to adjust the parameters $M$ and $N$ to achieve the suitable solution.  Figure \ref{f1} depicts $h$- and $p$-version methods in which the values of parameters $M$ and $N$ are equal to $1$, respectively. Figure \ref{f11} shows the $hp$-version collocation method for  each fixed $ N=1,2,4,8 $ when $ h_n=h=\frac{1}{N}$   and various values of $ M_n=M^* $ for $ n=1, \dots, N $.
 
Secondly, in order to compare the theoretical and the numerical solution, we consider $hp$-version with $M=2$ and various $N$. Therefore, it is expected to have a rate near $2$ by means of relation \eqref{ro}; namely, $\rho_N \approx m=2$. This expectation is experimentally verified and shown in the left sub-figure of Figure \ref{f0001}. 
 
 Finally, we consider different values for $\alpha$. According to Theorem \ref{thm9},  increasing the  values of $\alpha $  affirmatively affects on the convergence rate which is verified by the numerical results on the right sub-figure of Figure \ref{f0001}.
\end{example}
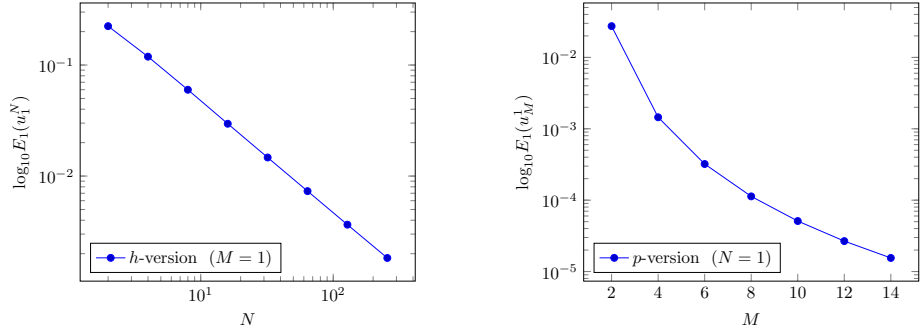
\begin{figure}
	\begin{center}
	\begin{tikzpicture}
	[scale=0.65, transform shape]
	\begin{axis}[
	xlabel={$ N $},ylabel={${\log}_{10}E_{1}(u_{1}^{N})$},xmode=log, ymode=log, legend entries={$ h\text{-version}~~(M=1)$}, legend pos=south west]
	\addplot coordinates {
		(2,2.24E-01)
		(4,1.19E-01)
		(8,5.99E-02)
		(16,2.96E-02)
		(32,1.47E-02)
		(64,7.32E-03)
    	(128,3.65E-03)
		(256,1.83E-3)
	};
		\end{axis}
	\end{tikzpicture}
	\hspace{1cm}
\begin{tikzpicture}
[scale=0.65, transform shape]
\begin{axis}[
xlabel={$ M $},ylabel={$ {\log}_{10}E_{1}(u_{M}^{1}) $}, ymode=log, legend entries={$p\text{-version}~~(N=1)$}, legend pos=south west]
\addplot coordinates {
	(2,2.74E-2)
	(4, 1.45E-3)
	(6,3.21E-4)
	(8,1.13E-4)
	(10, 5.10E-5)
	(12, 2.67E-5)
	(14, 1.55E-5)
};
\end{axis}
\end{tikzpicture}
\caption{Plots of the $ E_{1}(u_{M}^{N}) $ error in logarithmic scale for the   $ h$- and $p$-version collocation methods with $\alpha=0.3$ for Example \ref{ex111}.}%
\label{f1}
\end{center}
\end{figure}

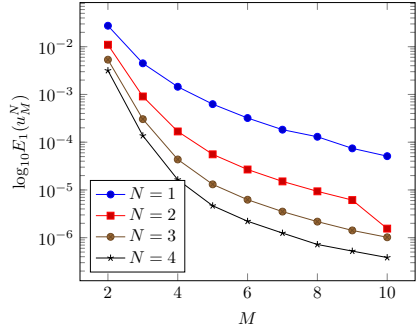
\begin{figure}
		\begin{center}
	\begin{tikzpicture}
	[scale=0.65, transform shape]
	\begin{axis}[
	xlabel={$M$},ylabel={${\log}_{10}E_{1}(u_{M}^{N})$}, ymode=log, legend entries={$N=1$, $N=2$, $N=3$, $N=4$}, legend pos=south west]
	\addplot coordinates {
		(2,2.75E-02)
		(3,4.49E-03)
		(4,1.45E-03)
		(5,6.27E-04)
		(6, 3.21E-04)
		(7,1.83E-04)
		(8,1.3E-04)
		(9,7.43E-05)
		(10,5.10E-05)
	};
	\addplot coordinates {
		(2,1.10E-02)
		(3,9.08E-04)
		(4,1.68E-04)
		(5, 5.57E-05)
		(6,2.68E-05)
		(7,1.51E-05)
		(8,9.35E-06)
		(9,6.13E-06)
		(10,1.55E-06)
	};
	\addplot coordinates {
		(2,5.35E-03)
		(3,3.04E-04)
		(4,4.35E-05)
		(5, 1.31E-05)
		(6,6.24E-06)
		(7,3.52E-06)
		(8,2.17E-06)
		(9,1.42E-06)
		(10,1.02E-6)
	};\addplot coordinates {
		(2,3.17E-03)
		(3,1.37E-04)
		(4,1.64E-05)
		(5, 4.67E-06)
		(6,2.21E-06)
		(7,1.25E-06)
		(8,7.21E-07)
		(9,5.24E-07)
		(10,3.84E-07)
	};
	\end{axis}
	\end{tikzpicture}
	\caption{Plots of the $ E_{1}(u_{M}^{N}) $ error in logarithmic scale for the $ hp$-version collocation method with $\alpha=0.3$ for Example \ref{ex111}.}%
	\label{f11}
			\end{center}
\end{figure}
\begin{figure}
	\begin{center}
	\begin{tikzpicture}
	[scale=0.65, transform shape]
	\begin{axis}[
	xlabel={Plot a)$~~~ N,~ \alpha=0.7 $},ylabel={${\log}_{10}E_{1}(u_{2}^{N})$},xmode=log, ymode=log, legend entries={$ M_n=M^*=1$, $\text{Theoretical results}$}, legend pos=south west]
	\addplot coordinates {
		(2,2.15E-02)
		(4,7.03E-03)
		(8,2.14E-03)
		(16,6.10E-04)
		(32,1.66E-04)
		(64,4.4E-05)
    	(128,1.10E-05)
		(256,2.66E-6)
	};
		\addplot coordinates {
			(2,5.00E-02)
			(256,3.052E-6)
		};
		\end{axis}
	\end{tikzpicture}
	\hspace{1cm}
\begin{tikzpicture}
[scale=0.65, transform shape]
\begin{axis}[
xlabel={Plot b)$ ~~~ M $},ylabel={$ N=2 $}, ymode=log, legend entries={$\alpha=0.3$,$\alpha=0.5 $,$\alpha=0.7 $}, legend pos=south west]
\addplot coordinates {
	(2,1.10E-2)
	(4, 1.68E-4)
	(6,2.68E-5)
	(8,9.35E-6)
	(10, 4.12E-6)
	(12, 2.43E-6)
};
\addplot coordinates {
(2,2.06E-2)
	(4, 1.73E-4)
	(6,1.65E-5)
	(8,4.77E-6)
	(10, 1.92E-6)
	(12, 9.47E-7)
};
\addplot coordinates {
(2,3.25E-2)
	(4,1.45E-4)
	(6,8.61E-6)
	(8,1.96E-6)
	(10, 7.02E-7)
	(12, 5.16E-7)
};
\end{axis}
\end{tikzpicture}
\caption{Plots of the $ E_{1}(u_{M}^{N}) $ error in logarithmic scale: a) Comparison between theoretical and numerical results b) The results of $hp$-method for Example \ref{ex111} for various $\alpha$.}%
\label{f0001}
\end{center}
\end{figure}
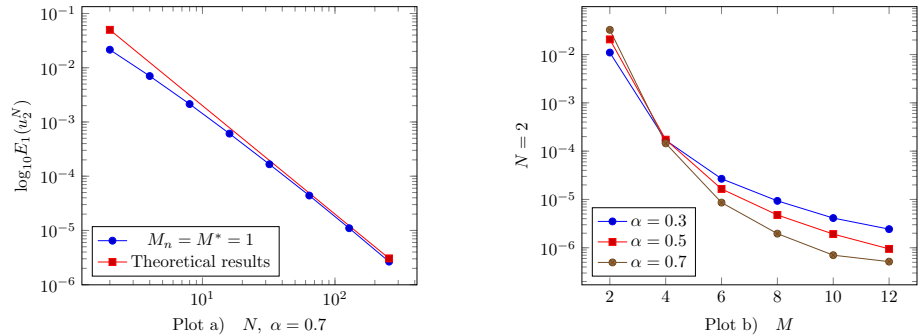
\begin{example} \label{explato}
(\cite{plato12})	In the following example, we consider solving the linear Abel equation 
	\begin{equation*}
	\int_{0}^t (t-s)^{-0.5}\exp(-t+s)u(s)\mathrm{d}s=f(t), \quad t\in [0,T],
	\end{equation*}
	where $f(t)=\exp(-t)(t^4+t^6)$ and the exact solution is $u(t)= \exp(-t)\big(\frac{4!}{\Gamma(4.5)}t^{3.5}+\frac{6!}{\Gamma(6.5)}t^{5.5}\big). $

	For the sake of good comparison, let us take into account some assumptions considered in the interesting paper \cite{plato12}. Numerical experiments with $M=2$ and step sizes $h=1/2^q, q=5,6, \dots, 11$ accompanied with the noise level $\delta=h^{2.5}$ for investigating the effect of perturbation are employed.	The perturbation is added to the right hand side as $f^\delta(t)=f(t)+\delta$. Let us denote $E_2^\delta(u_{2}^{N}):=\max_{t\in \Omega}\vert u_N^\delta(t)-u(t)\vert$. Table 1 demonstrates the superiority of the $hp$-collocation method against trapezoidal method. The last column shows that the order of perturbation symbol $\delta$ is not even linear. This importance verifies the well-posedness of the scheme.
	 
	On the other hand, the purpose of adaptive $hp$-collocation method is to utilize subspaces with lower dimension which yield less computational complexity and CPU time. For instance, if we take $M=12$ and $N=2$, then the absolute error is $2.6 \textrm{e}{-07}$ for $L=24$ whereas the best results of  Table 1 are achieved by $L=2\times 2048$.  
	\begin{table}[ht]\label{table1}
			\begin{center}
				\begin{small}
					\caption{The comparison of trapezoidal method \cite{plato12} and $hp$-collocation method  for different $ N $ with fixed $ M=M_n+1=2$ and $h=1/2^q, q=5,6,\dots,11$ in the sense of $E_N^\delta$ for Example \ref{explato}.}\vspace*{0.1in}
					\begin{tabular}{c|c|cc|cc|c|c}
						\hline
						\noalign{\smallskip}
						$N$&$\delta $ && $E_N^\delta$ 
						(\cite{plato12})
						 && $E_N^\delta$ ($ hp$) &$E_N^\delta$(\cite{plato12})/$\delta^{0.8}$&$E_N^\delta$($hp$)$/\delta^{0.85}$
						\\\noalign{\smallskip}
						\hline
						\noalign{\smallskip}\hline\noalign{\smallskip}
						$32$&$1.7\textrm{e}-04$&&$2.95\textrm{e}-03$&&$ 7.29\textrm{e}-04$&$3.02 $& $ 1.17$\\						$64$&$3.1\textrm{e}-05$&&$1.04\textrm{e}-03$&&$1.83\textrm{e}-04$&$4.26$& $ 1.24$\\
						$128$&$5.4\textrm{e}-06$&&$2.03\textrm{e}-04$ &&$ 4.57\textrm{e}-05 $&$3.32 $& $ 1.37$\\
						$256$ &$9.5\textrm{e}-07$&&$5.79\textrm{e}-05$&&$1.14\textrm{e}-05 $&$3.79 $& $ 1.50$\\
						$512$&$1.7\textrm{e}-07$&&$1.72\textrm{e}-05$&&$2.84\textrm{e}-06 $&$4.50 $&$ 1.61$\\
						$1024$&$3.0\textrm{e}-08$&&$4.24\textrm{e}-06$ &&$ 7.09\textrm{e}-07 $&$4.45 $&$1.75  $\\
						$2048$ &$5.3\textrm{e}-09$&&$1.09\textrm{e}-06$&&$1.99\textrm{e}-07$&$4.58 $&$2.15  $\\
						\hline
			$\rho_N$&&&$1.50$&&$ 1.99 $&&\\
						\noalign{\smallskip}\hline
					\end{tabular}
				\end{small}
			\end{center}
			\label{tab:1}
	\end{table}
	\end{example}
\subsubsection*{Smooth solution}
\begin{example}(\cite{branca1978})\label{ex101}
	In this example, we apply the method to the following nonlinear weakly singular Volterra integral equation of the first kind
	\begin{equation*}
	\int_{0}^t (t-s)^{-0.5} \big(1+s+tu(s)\big)u(s)\mathrm{d}s=\frac{32}{45045}(1287+1144+960t^4)t^{3.5}, \quad t\in [0,1],
	\end{equation*}
	with the exact solution $ u(t)=t^3$. 
	Table 2 reports the comparison of  Finite Difference Method (FDM) of the third order \cite{branca1978} and $ hp$-collocation method with the same value of $ L $. The present scheme runs for  various values of $ N$ with fixed step size $ h_n=h=\frac{1}{N} $, uniform mode $ M=M_n+1=3$ for $ n=1, \dots ,N$. As expected from (\ref{ro}), $ \rho_N $ is approximately equal to $ m\leq M_{\min}+1=3 $. 
		
	 Adaptivity and capability of the present scheme to obtain the best result are provided according to Remark \ref{rem3}. We take our desired absolute error equal to $10^{-14}$ and hence achieve the appropriate solution  with the absolute error $2.33e-15$,  when $M$ and $N$ are chosen $4$ and $1$, respectively. As we expected, $p$-version works well for the problems with smooth solution.
\end{example} 
\begin{table}
		\begin{center}
			\begin{small}
				\caption{The comparison of FDM  \cite{branca1978} and $hp$-collocation method  for different $ N $ with fixed $ M=M_n+1=3$ and $h=h_n=\frac{1}{N}$ for $ n=1, \dots ,N$ in terms of $E_{2}(u_{3}^{N})$ for Example \ref{ex101}.}\vspace*{0.1in}
				\begin{tabular}{ccccccc}
					\hline
					\noalign{\smallskip}
					$N$&& FDM \cite{branca1978} && $ hp-$collocation
					\\\noalign{\smallskip}
					\hline
					\noalign{\smallskip}\hline\noalign{\smallskip}
					$10$&&$3.7\textrm{e}-04$&& $ 9.42\textrm{e}-05$\\
					$20$&&$4.7\textrm{e}-05$&& $ 1.08\textrm{e}-05$\\
					$40$&&$6.0\textrm{e}-06$ && $ 1.12\textrm{e}-06$\\
					$80$ &&$7.6\textrm{e}-07$&& $ 1.18\textrm{e}-07$\\
					$160$&&$9.6\textrm{e}-08$&&$ 1.21\textrm{e}-08$\\
					$320$&&$1.2\textrm{e}-08$ &&$1.32\textrm{e}-09  $\\
					\hline
					$\rho_N$&&$2.97$&&$ 3.12 $\\
					\noalign{\smallskip}\hline
				\end{tabular}
			\end{small}
		\end{center}
		\label{tab:2}
\end{table}
\begin{example} 
	(\cite{Liu}) Consider the following linear Abel equation
\begin{equation*}
\int_{0}^t (t-s)^{-0.25}(t^2s^3+s^4+1)u(s)\mathrm{d}s= \dfrac{128 t^{\frac{11}{4}} (3933+256t^4(8+9t))}{908523}, \quad t\in [0,1],
\end{equation*}		
with smooth solution $ u(t)=t^2. $ 
This example is studied in \cite{Liu} using mechanical quadrature method and its extrapolation by converting the above first kind integral equation to the second kind.
 The lowest error in terms of absolute error utilizing $ h^2$-extrapolation and $ N=80 $ is reported  $ 1.72\textrm{e}-8$ in \cite{Liu}. As described in the previous example, we expect  $ p$-version to work well for such a smooth case. The appropriate approximate solution using adaptive $ hp$-collocation for the given tolerance  $ 10^{-15}$ is  achieved by $ M=4$ and $ N=1$; namely $E_{2}(u_{4}^{1}) = 3.33\textrm{e}-16$. The superiority of the proposed method in the sense of computational complexity and accuracy is evident.
\end{example}
\subsubsection*{Discontinuous solution.}
\begin{example} \label{ex5}
	In the following example, we consider solving the nonlinear Abel equation
	\begin{equation*}
	\int_{0}^t (t-s)^{-0.2}\kappa(t,s)u^5(s)\mathrm{d}s=f(t), \quad t\in [0,1],
	\end{equation*}
	where $ \kappa(t,s)=\sin(t-s) $ and $f(t)$ is chosen such that 
	\begin{equation*}
	\begin{split}
	u(t)=\left\lbrace \begin{array}{cc}
\exp(-t),&0\leq t<0.5,\\
2-t^2,&   0.5 \leq t\leq 1,
	\end{array}\right.
	\end{split}
	\end{equation*}
be the exact solution. 
Note that in this example, $ \kappa(t,t)=0$ which demonstrates that the problem can not be converted into the second kind Abel equation. Secondly, the exact solution $ u(x) $ is a discontinuous function which  obviously can not be solved by $ p$-version methods.
Figure \ref{f7} shows  considerable results  for various $ M $ with fixed step size $h_n=h=\frac{1}{2} $ and $M_n=M^*$ for $n=1,2$. 
Here, we take the degree of polynomials for each  sub-interval $\Omega_n$	with $M_n=M^*$ for all $n =1,2$. Now, we implement the program in order to reach the best result of the scheme automatically. 	The best finding belongs to different degree mode $M_1=9$ and $M_2=4$ for the first and second sub-intervals. Using only $15$ basis functions leads to achieve an appropriate solution with the $L^2$-error norm $3.02\textrm{e}-8$. The best reported error in Figure \ref{f7} was for $M_1=M_2=9$ with the value $2.12\textrm{e}-7$. A comparison between these two results show the adaptivity of the scheme in which a lower dimension of basis functions prevent the scheme from aggregating more errors.
\end{example}
\begin{figure}
	\begin{center}
	\begin{tikzpicture}
	[scale=0.65, transform shape]
	\begin{axis}[
	xlabel={$ M $},ylabel={$ E_1^{N}(u_{M}^{2}) $}, legend entries={$N=2$}, ymode=log]
	\addplot coordinates {
		(2,3.67E-2)
		(4, 4.97E-4)
		(6,2.16E-5)
		(8,3.34E-6)
		(10, 2.12E-7)
	};
	\end{axis}
	\end{tikzpicture}
	\caption{Plots of the $ E_1^{N}(u_{M}^{2}) $ error in logarithmic scale for  different mode $M$ with fixed $ h_n=h=\frac{1}{2}$ and $M_n=M^*$ for $n=1,2$ for Example \ref{ex5}.}%
	\label{f7}%
	\end{center}
\end{figure}
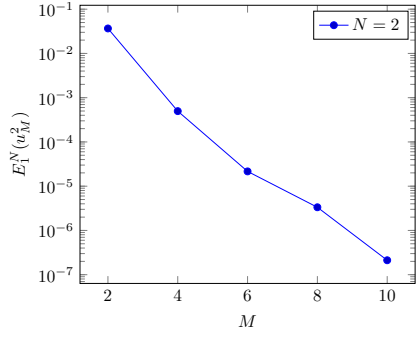
\subsubsection*{Unknown exact solution}
	In the following example, we consider equation  which has a unique solution according Theorem \ref{thm1}.
	\begin{example}\label{exunknown2}
		The following nonlinear weakly singular integral equation is considered
		\begin{equation*}
		\int_{0}^{t}  (t-s)^{-0.4}\kappa(t,s)\Big(\cos(2s  u(s))-\ln u^3(s)-\sqrt{s u(s)}\Big)\mathrm{d}s=t^{1.5}-\vert t\vert,\quad t\in [0,1.5],
		\end{equation*}
		where 
			\begin{equation*}
			\begin{split}
			\kappa(t,s)=\left\lbrace \begin{array}{cc}
		t^2-s+5,& 0<s<0.5,\quad 0<t<1.5,\\
		\exp(st)+\frac{4}{s+1}-2,& 0.5<s<1,\quad 0<t<1.5,\\
			\frac{t}{s},& 1<s<1.5,\quad 0<t<1.5.
			\end{array}\right.
			\end{split}
			\end{equation*}
	The exact solution in unknown, therefore we choose $ u^3_{12}(t) $ with $ L=12\times 3=36 $ basis functions as a benchmark for comparison.   Figure 6
		  shows the benchmark which is the approximate solution for $ T=1.5 $.
	 Figure \ref{f4}
	depicts the convergence of the scheme to benchmark solution by increasing $ M$ and $ N $ with $M=M_n+1=M^*+1  $ and fixed step size $h_n=h=\frac{1}{N} $ for $ n=1, \dots, N $.  
	 It is conspicuous that $hp$-method works well for approximating these non-smooth solution against $p$-version. 
	\end{example}
	\begin{figure}
		\begin{center}
		\begin{tikzpicture}
		[scale=0.65, transform shape]
		\begin{axis}[
		xlabel={$ M $},ylabel={${\log}_{10}E_{2}(u^N_{M}(t))$}, ymode=log, legend entries={$ N=1 $,$N=2$,$N=3$,$N=6$}, legend pos=south west]
		\addplot coordinates {
			(2,5.76E-02)
			(4,3.38E-02)
			(6, 2.24E-02)
			(8, 9.42E-03)
			(10,9.21E-03)
		};
		\addplot coordinates {
			(2,6.37E-02)
			(4,2.92E-02)
			(6, 8.30E-03)
			(8, 8.27E-03)
			(10,8.29E-03)
		};
		\addplot coordinates {
			(2,5.30E-02)
			(4,5.91E-03)
			(6, 1.71E-03)
			(8, 9.72E-04)
			(10,2.62E-04)
		};
			\addplot coordinates {
				(2,2.05E-02)
				(4,1.95E-03)
				(6, 8.82E-04)
				(8, 2.54E-04)
				(10,8.88E-05)
			};
		\end{axis}
		\end{tikzpicture}
		\caption{ Plots of the $E_{2}(u^N_{M}(t))$ in logarithmic scale for Example \ref{exunknown2}. }%
		\label{f4}%
		\end{center}
	\end{figure}
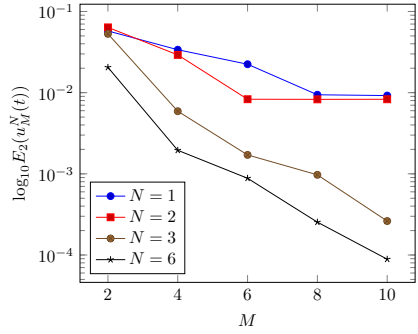
	\begin{figure}\label{pic7}
		\begin{center}			\resizebox*{7cm}{!}{\includegraphics{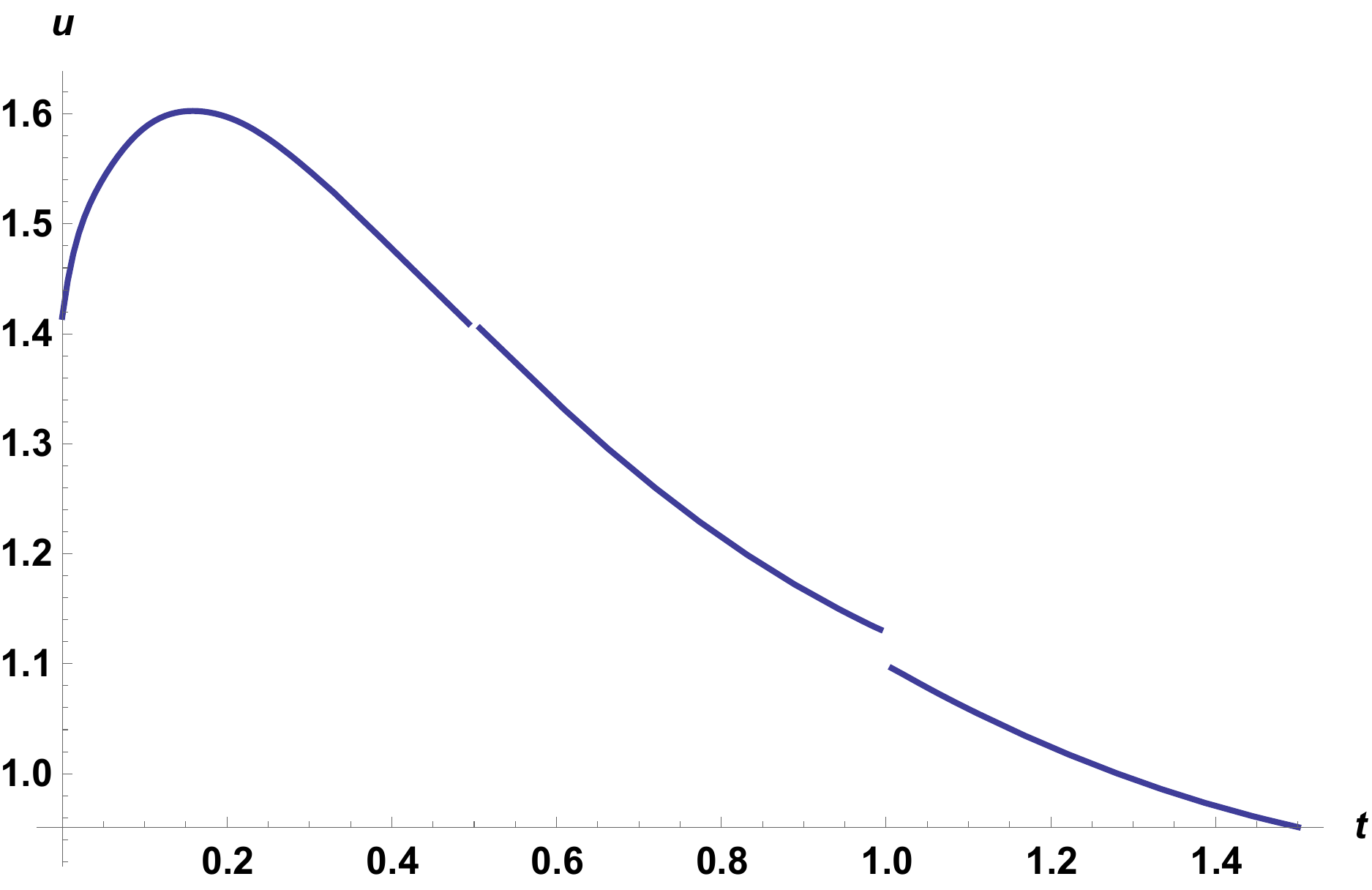}}
			\caption{The approximate solution $ u^3_{12}(t) $ for $ T=1.5$. }
		\end{center}
	\end{figure}
\section*{Conclusion\label{SecConcl}}
The first kind integral equations and their approximations are interesting problems from both theory and application view  points.  This paper concerns a nonlinear class of them so-called Abel integral equations in a general form. The existence and uniqueness of the solution have been investigated in the suitable Sobolev spaces under some  assumptions. The  $hp$-version Jacobi projection methods have been studied and a prior error analysis in $L^2$-norm is developed for Abel integral equations.  Numerical results indicate that the proposed scheme is effective and powerful to deal with smooth and non-smooth solutions. This analysis can be extended for the nonlinear fractional differential operators of Riemann-Liouville and Caputo types which are being considered and investigated in our future works. 

\section*{Appendix A.}\label{appin}
In this section, we are acquiring error bounds for the terms $ \Vert E_1\Vert_{\Omega_n}$ and $ \Vert E_2\Vert_{\Omega_n}$.
In advance,
\begin{equation}
\begin{aligned}
\Vert E_2 \Vert^2_{\Omega_n}&=\int_{\Omega_n}(\frac{t-t_{n-1}}{h_n})^{2\alpha}\Big(\int_{\Omega_n}(t_n-\lambda)^{\alpha-1}(\mathcal{I}_{\lambda,M_n}^{\alpha-1,0}-\mathcal{I})\big(\kappa(\sigma(\lambda,t),t)\psi(\sigma(\lambda,t),u_{M_n}^n(\sigma(\lambda,t)))\big)\mathrm{d}\lambda\Big)^2\mathrm{d}t\\& \leq \Vert (\dfrac{t-t_{n-1}}{h_n})^{2\alpha}\Vert_{\infty}\Big\vert\int_{\Omega_n}\Big(\int_{\Omega_n}(t_n-\lambda)^{\alpha-1}(\mathcal{I}_{\lambda,M_n}^{\alpha-1,0}-\mathcal{I})\big(\kappa(\sigma(\lambda,t),t)\psi(\sigma(\lambda,t),u_{M_n}^n(\sigma(\lambda,t)))\big)\mathrm{d}\lambda\Big)^2\mathrm{d}t\Big\vert\\
&\leq c\Big(\int_{\Omega_n} (t_n-\lambda)^{\alpha-1} \mathrm{d}\lambda\Big)\int_{\Omega_n}\int_{\Omega_n}(t_n-\lambda)^{\alpha-1}\Big\vert (\mathcal{I}_{\lambda,M_n}^{\alpha-1,0}-\mathcal{I})\Big(\kappa(\sigma(\lambda,t),t)\psi(\sigma(\lambda,t),u_{M_n}^n(\sigma(\lambda,t)))\Big)\Big\vert^2\mathrm{d}\lambda\mathrm{d}t,\\&\leq c_\alpha h_n^{\alpha}\int_{\Omega_n}\int_{\Omega_n}(t_n-\lambda)^{\alpha-1}\Big\vert (\mathcal{I}_{\lambda,M_n}^{\alpha-1,0}-\mathcal{I})\Big(\kappa(\sigma(\lambda,t),t)\psi(\sigma(\lambda,t),u_{M_n}^n(\sigma(\lambda,t)))\Big)\Big\vert^2\mathrm{d}\lambda\mathrm{d}t,
\end{aligned}
\end{equation}
where the above inequalities follow directly from the Cauchy-Schwartz inequality. Now using Lemma \ref{lem1} and the relation \eqref{tau}, we get 
\begin{equation}
\begin{aligned}
\Vert E_2 \Vert_{\Omega_n}^2&\leq c_\alpha h_n^{2m+\alpha}(M_n+1)^{-2m}\int_{\Omega_n}\int_{\Omega_n}(t_n-\lambda)^{\alpha-1}
\Big\vert\partial_\lambda^{m}\Big(\kappa(\sigma(\lambda,t),t)\psi\big(\sigma(\lambda,t),u_{M_n}^n(\sigma(\lambda,t))\big)\Big)\Big\vert^2\mathrm{d}\lambda\mathrm{d}t\\
&\leq c_\alpha h_n^{2m+\alpha}(M_n+1)^{-2m}\int_{\Omega_n}\int_{t_{n-1}}^t  (h_n \,.\frac{t-\tau}{t-t_{n-1}})^{\alpha-1}(\frac{t-t_{n-1}}{h_n})^{2m-1}\Big\vert\partial_\tau^{m}\Big(\kappa(\tau,t)\psi(\tau,u_{M_n}^n(\tau))\Big)\Big\vert^2\mathrm{d}\tau\mathrm{d}t\\ \notag
&\leq
c_\alpha h_n^{2m+2\alpha}(M_n+1)^{-2m}\sum\limits_{i=0}^{m}\int_{\Omega_n} \Big\vert\partial_\tau^{m}\Big(\psi(\tau,u_{M_n}^n(\tau))\Big)\Big\vert^2\mathrm{d}\tau\\\notag
&=c_\alpha h_n^{2m+2\alpha}(M_n+1)^{-2m} \Vert \psi(.,u_{M_n}^n(.))\Vert^2_{H^m_{\chi^{\alpha-1,0}}(\Omega_n)}.
\end{aligned}
\end{equation}
In order to find an upper bound for the term $ E_1$, we notice that 

\begin{equation}\begin{split}\label{e0}
\Vert E_1\Vert_{\Omega_n}^2=&\Big \Vert  (\mathcal{I}_{M_n}^{x}-\mathcal{I})\Big(\int_{t_{n-1}}^{t}(t-\tau)^{\alpha-1}\mathcal{I}_{\tau,M_n}^{\alpha-1,0}\big(\kappa(\tau,t)\psi\big(\tau,u_{M_n}^n(\tau)\big)\mathrm{d}\tau \Big)\Big\Vert^2_{\Omega_n}\\\notag
=&\Big \Vert \int_{t_{n-1}}^{t}(t-\tau)^{\alpha-1}\mathcal{I}_{\tau,M_n}^{\alpha-1,0}\big(\kappa(\tau,t)\big[\psi\big(\tau,u_{M_n}^n(\tau)\big)-\psi\big(\tau,u^n(\tau)\big)\big]\big)\mathrm{d}\tau  \\\notag&-\mathcal{I}_{M_n}^{t}\Big(\int_{t_{n-1}}^{t}(t-\tau)^{\alpha-1}\mathcal{I}_{\tau,M_n}^{\alpha-1,0}\big(\kappa(\tau,t)\big[\psi\big(\tau,u_{M_n}^n(\tau)\big)-\psi\big(\tau,u^n(\tau)\big)\big]\mathrm{d}\tau \Big)
\\\notag&+(\mathcal{I}-\mathcal{I}_{M_n}^t)\Big(\int_{t_{n-1}}^t (t-\tau)^{\alpha-1}(\mathcal{I}_{\tau,M_n}^{\alpha-1,0}-\mathcal{I})\big(\kappa(\tau,t)\psi(\tau,u^n(\tau))\mathrm{d}\tau\Big)
\\\notag&
+(\mathcal{I}-\mathcal{I}_{M_n}^t)\big(\int_{t_{n-1}}^t (t-\tau)^{\alpha-1}\kappa(\tau,t)\psi(\tau,u^n(\tau))\mathrm{d}\tau\big)\Big\Vert^2
\\ \notag\leq &(E_{1,1}+E_{1,2}+E_{1,3}+\Vert B_4\Vert^2),
\end{split}
\end{equation}
where the above terms $ E_{1,i}, i=1, 2, 3 $ will introduce in sequel with their relevant upper bounds and the error bound for $ \Vert B_4\Vert^2 $ is derived in \eqref{b4}. Using Eqs. \eqref{change}, \eqref{change2} and Cauchy-Schwartz inequality, we get 
\begin{equation}
\begin{split}
E_{1,1}=&\Big \Vert \int_{t_{n-1}}^{t}(t-\tau)^{\alpha-1}\mathcal{I}_{\tau,M_n}^{\alpha-1,0}\big(\kappa(\tau,t)\big[\psi\big(\tau,u_{M_n}^n(\tau)\big)-\psi\big(\tau,u^n(\tau)\big)\big]\big)\mathrm{d}\tau  \Big \Vert_{L^2(\Omega_n)}^2
\\	&= \int_{\Omega_n}\Big\vert \int_{t_{n-1}}^{t}(t-\tau)^{\alpha-1}\mathcal{I}_{\tau,M_n}^{\alpha-1,0}\big(\kappa(\tau,t)\big[\psi\big(\tau,u_{M_n}^n(\tau)\big)-\psi\big(\tau,u^n(\tau)\big)\big]\big)\mathrm{d}\tau\Big\vert^2 \mathrm{d}t
\\&\leq \int_{\Omega_n}\Big(\int_{t_{n-1}}^{t}(t-\tau)^{\alpha-1}\mathrm{d}\tau\Big) \Big(\int_{t_{n-1}}^{t}(t-\tau)^{\alpha-1}\Big(\mathcal{I}_{\tau,M_n}^{\alpha-1,0}\big(\kappa(\tau,t)\big[\psi\big(\tau,u_{M_n}^n(\tau)\big)-\psi\big(\tau,u^n(\tau)\big)\big]\Big)^2\mathrm{d}\tau\Big) \mathrm{d}t
\\& \leq c_\alpha h_n^\alpha \int_{\Omega_n}\Big(\int_{t_{n-1}}^{t}(t-\tau)^{\alpha-1}\Big(\mathcal{I}_{\tau,M_n}^{\alpha-1,0}\big(\kappa(\tau,t)\big[\psi\big(\tau,u_{M_n}^n(\tau)\big)-\psi\big(\tau,u^n(\tau)\big)\big]\Big)^2\mathrm{d}\tau\Big) \mathrm{d}t
\\& \leq c_\alpha h_n^\alpha \int_{\Omega_n} (\dfrac{t-t_{n-1}}{2})^\alpha \sum_{j=0}^{M_n}\Big(\kappa(\tau_{n,j},t)\big[\psi\big(\tau_{n,j},u_{M_n}^n(\tau_{n,j})\big)-\psi\big(\tau_{n,j},u^n(\tau_{n,j})\big)\big]\Big)^2 w_{n,j}\,\mathrm{d}t.
\end{split}
\end{equation}
Moreover, due to the Lipschitz condition for the function $ \psi $ with respect to second variable, we obtain 
\begin{equation}
\begin{split}\label{e1}
E_{1,1}& \leq c_\alpha h_n^\alpha \gamma^2\int_{\Omega_n} (\dfrac{t-t_{n-1}}{2})^\alpha \sum_{j=0}^{M_n}\big[u_{M_n}^n(\tau_{n,j})-u^n(\tau_{n,j})\big]^2 w_{n,j}\mathrm{d}t
\\& \leq c_\alpha h_n^\alpha \gamma^2\int_{\Omega_n} \int_{t_{n-1}}^t (t-\tau)^{\alpha-1}\big[\mathcal{I}_{\tau,M_n}^{\alpha-1,0}\big(u_{M_n}^n(\tau)-u^n(\tau)\big)\big]^2 \mathrm{d}\tau\mathrm{d}t
\\& \leq c_\alpha h_n^\alpha \gamma^2
\int_{\Omega_n}\int_{t_{n-1}}^{t}\Big[(\mathcal{I}_{\tau,M_n}^{\alpha-1,0} u^n(\tau)-u^n(\tau))^2+(u^n(\tau)-u_{M_n}^n(\tau))^2\Big] (t-\tau)^{\alpha-1}\mathrm{d}\tau\mathrm{d}t
\\& \leq c_\alpha h_n^{\alpha} \gamma^2\big(h_n^{\alpha} \Vert e_n\Vert_{\Omega_n}^2+h_n^{2m+1}(M_n+1)^{-2m}\Vert \partial_t^{m}u^n\Vert^2_{L^2_{\chi^{\alpha-1,0}}(\Omega_n)}\big)
\\& \leq c_\alpha h_n^{\alpha}\gamma^2 \big(h_n^{\alpha} \Vert e_n\Vert_{\Omega_n}^2+h_n^{2m+1}(M_n+1)^{-2m}\Vert u^n\Vert^2_{H^m_{\chi^{\alpha-1,0}}(\Omega_n)}\big).
\end{split}
\end{equation}
In order to seek an upper bound for $ E_{1,2}$, we follow Eqs. \eqref{hn}, \eqref{change}, \eqref{change2}, Lipschitz condition and H\"{o}lder inequality to get that
\begin{equation}
\begin{aligned}
E_{1,2}&=\Big \Vert	\mathcal{I}_{M_n}^{t}\Big(\int_{t_{n-1}}^{t}(t-\tau)^{\alpha-1}\mathcal{I}_{\tau,M_n}^{\alpha-1,0}\big(\kappa(\tau,t)\big[\psi\big(\tau,u_{M_n}^n(\tau)\big)-\psi\big(\tau,u^n(\tau)\big)\big]\mathrm{d}\tau \big)\big)\Big\Vert
\\&=\int_{\Omega_n}	\Big[\mathcal{I}_{M_n}^{t}\Big(\int_{t_{n-1}}^{t}(t-\tau)^{\alpha-1}\mathcal{I}_{\tau,M_n}^{\alpha-1,0}\big(\kappa(\tau,t)\big[\psi\big(\tau,u_{M_n}^n(\tau)\big)-\psi\big(\tau,u^n(\tau)\big)\big]\big)\mathrm{d}\tau \Big]^2\mathrm{d}t
\\&=\frac{h_n}{2}\sum_{j=0}^{M_n}\Big(\int_{t_{n-1}}^{t_{n,j}}(t_{n,j}-\tau)^{\alpha-1}\mathcal{I}_{\tau,M_n}^{\alpha-1,0}\big(\kappa(\tau,t_{n,j})\big[\psi\big(\tau,u_{M_n}^n(\tau)\big)-\psi\big(\tau,u^n(\tau)\big)\big]\mathrm{d}\tau\Big)^2 w_{n,j}
\\&\leq ch_n \sum_{j=0}^{M_n}\big(\int_{t_{n-1}}^{t_{n,j}}(t_{n,j}-\tau)^{\alpha-1}\mathrm{d}\tau\big)\big(\int_{t_{n-1}}^{t_{n,j}}(t_{n,j}-\tau)^{\alpha-1}\Big(\mathcal{I}_{\tau,M_n}^{\alpha-1,0}\big(\kappa(\tau,t_{n,j})\big[\psi\big(\tau,u_{M_n}^n(\tau)\big)-\psi\big(\tau,u^n(\tau)\big)\big]\big)\Big)^2\mathrm{d}\tau\big) w_{n,j}
\\& \leq c_\alpha h_n^{\alpha+1}\sum_{j=0}^{M_n}(\dfrac{t_{n,j}-t_{n-1}}{2})^{\alpha}\sum_{i=0}^{M_n}\Big(\kappa(\tau_{n,i},t_{n,j})\big[\psi\big(\tau_{n,i},u_{M_n}^n(\tau_{n,i})\big)-\psi\big(\tau_{n,i},u^n(\tau_{n,i})\big)\big]\Big)^2 w_{n,i} w_{n,j}
\\&\leq c_\alpha h_n^{\alpha+1}\gamma^2 \sum_{j=0}^{M_n}(\dfrac{t_{n,j}-t_{n-1}}{2})^{\alpha}\sum_{i=0}^{M_n} \Big \vert u_{M_n}^n(\tau_{n,i})-u^n(\tau_{n,i}) \Big \vert^2 w_{n,i}\,w_{n,j}
\\& \leq c_\alpha h_n^{\alpha+1} \gamma^2 \sum_{j=0}^{M_n}w_{n,j}\Big((\dfrac{t_{n,j}-t_{n-1}}{2})^{\alpha}\sum_{i=0}^{M_n} \Big \vert u_{M_n}^n(\tau_{n,i})-u^n(\tau_{n,i}) \Big \vert^2 w_{n,i} \Big).
\end{aligned}
\end{equation}
Since $ \sum\limits_{j=0}^{M_n}w_{n,j}=2 $ and  Eq. \eqref{change2} and Lemma \ref{lem1}, the above inequality can be simplified as
\begin{equation}
\begin{aligned}\label{e2}
E_{1,2}&\leq  c_\alpha h_n^{\alpha+1}\gamma^2 \int_{t_{n-1}}^{t_{n,j}}\Big(\mathcal{I}_{\tau,M_n}^{\alpha-1,0} u^n(\tau)-u_{M_n}^n(\tau) \Big)^2 (t_{n,j}-\tau)^{\alpha-1}\mathrm{d}\tau\\
&\leq c_\alpha h_n^{\alpha+1}\gamma^2 \int_{t_{n-1}}^{t_{n,j}}\Big[(\mathcal{I}_{\tau,M_n}^{\alpha-1,0} u^n(\tau)-u^n(\tau))^2+(u^n(\tau)-u_{M_n}^n(\tau))^2\Big] (t_{n,j}-\tau)^{\alpha-1}\mathrm{d}\tau\\
& \leq c_\alpha h_n^{\alpha}\gamma^2 \big(h_n^{\alpha} \Vert e_n\Vert^2 + h_{n}^{2m+1}(M_n+1)^{-2m}
\Vert \partial_t^{m}u^n\Vert^2_{L^2_{\chi^{\alpha-1,0}}(\Omega_n)}\big)\\
& \leq c_\alpha h_n^{\alpha}\gamma^2 \big(h_n^{\alpha}  \Vert e_n\Vert_{\Omega_n}^2 +h_{n}^{2m+1}(M_n+1)^{-2m}\Vert u^n\Vert^2_{H^m_{\chi^{\alpha-1,0}}(\Omega_n)}\big).	
\end{aligned}
\end{equation}
Next, we derive an estimation for $ E_{1,3}.$
\begin{equation}
\begin{aligned}
E_{1,3}&= \Vert(\mathcal{I}-\mathcal{I}_{M_n}^t)\big(\int_{t_{n-1}}^t (t-\tau)^{\alpha-1}(\mathcal{I}_{\tau,M_n}^{\alpha-1,0}-\mathcal{I})\big(\kappa(\tau,t)\psi(\tau,u^n(\tau))\mathrm{d}\tau\big)\Vert_{\Omega_n}^2
\\
&=\int_{\Omega_n}\Big((\mathcal{I}-\mathcal{I}_{M_n}^t)\big(\int_{t_{n-1}}^t (t-\tau)^{\alpha-1}(\mathcal{I}_{\tau,M_n}^{\alpha-1,0}-\mathcal{I})\big(\kappa(\tau,t)\psi(\tau,u^n(\tau))\mathrm{d}\tau\big)\Big)^2\mathrm{d}t\\
&=\int_{\Omega_n}\Big((\mathcal{I}-\mathcal{I}_{M_n}^t)(\dfrac{t-t_{n-1}}{h_n})^{\alpha}\int_{\Omega_n} (t_n-\lambda)^{\alpha-1}(\mathcal{I}_{\lambda,M_n}^{\alpha-1,0}-\mathcal{I})\big(\kappa(\sigma(\lambda,t),t)\psi(\sigma(\lambda,t),u^n(\sigma(\lambda,t)))\big)\mathrm{d}\lambda\Big)^2\mathrm{d}t\\
&=\int_{\Omega_n}\Big((\mathcal{I}-\mathcal{I}_{M_n}^t)\int_{\Omega_n} (t_n-\lambda)^{\alpha-1}(\mathcal{I}_{\lambda,M_n}^{\alpha-1,0}-\mathcal{I})\big((\dfrac{t-t_{n-1}}{h_n})^{\alpha}\kappa(\sigma(\lambda,t),t)\psi(\sigma(\lambda,t),u^n(\sigma(\lambda,t)))\big)\mathrm{d}\lambda\Big)^2\mathrm{d}t.
\end{aligned}
\end{equation}
According to Lemma \ref{lemf} with $ m=1 $, we get that
\begin{equation}\label{ee}
\begin{split}
E_{1,3}\leq&ch_n^{2}(M_n+1)^{-2}\int_{\Omega_n}\Big(\int_{\Omega_n} (t_n-\lambda)^{\alpha-1}(\mathcal{I}_{\lambda,M_n}^{\alpha-1,0}-\mathcal{I})\Big[\partial_t\Big((\dfrac{t-t_{n-1}}{h_n})^{\alpha}\\&\times\kappa(\sigma(\lambda,t),t)\psi(\sigma(\lambda,t),u^n(\sigma(\lambda,t)))\Big)\Big]\mathrm{d}\lambda\Big)^2\mathrm{d}t\\\leq& ch_n^2(M_n+1)^{-2}\int_{\Omega_n}\Big[\partial_t((\dfrac{t-t_{n-1}}{h_n})^{2\alpha})\Big(\int_{\Omega_n} (t_n-\lambda)^{\alpha-1}(\mathcal{I}_{\lambda,M_n}^{\alpha-1,0}-\mathcal{I})\\
&~~~\times\big(\kappa(\sigma(\lambda,t),t)\psi(\sigma(\lambda,t),u^n(\sigma(\lambda,t)))\big)\mathrm{d}\lambda\Big)^2\\&~~~+(\dfrac{t-t_{n-1}}{h_n})^{2\alpha}\Big(\int_{\Omega_n} (t_n-\lambda)^{\alpha-1}(\mathcal{I}_{\lambda,M_n}^{\alpha-1,0}-\mathcal{I})\Big(\partial_t\big(\kappa(\sigma(\lambda,t),t)\psi(\sigma(\lambda,t),u^n(\sigma(\lambda,t)))\big)\Big)\mathrm{d}\lambda\Big)^2\Big]\mathrm{d}t\\
\leq& c h^2_{n}(M_n+1)^{-2}(D_{3,1}+D_{3,2}+D_{3,3}),
\end{split}
\end{equation}
where 
\begin{equation}
\begin{aligned}
D_{3,1}=&\int_{\Omega_n}(\partial_t(\dfrac{t-t_{n-1}}{h_n})^{2\alpha})\Big(\int_{\Omega_n} (t_n-\lambda)^{\alpha-1}(\mathcal{I}_{\lambda,M_n}^{\alpha-1,0}-\mathcal{I})\big(\kappa(\sigma(\lambda,t),t)\psi(\sigma(\lambda,t),u^n(\sigma(\lambda,t)))\big)\Big)^2\mathrm{d}\lambda\mathrm{d}t,\\
D_{3,2}=&\int_{\Omega_n}(\dfrac{t-t_{n-1}}{h_n})^{2\alpha}\Big(\int_{\Omega_n} (t_n-\lambda)^{\alpha-1}(\mathcal{I}_{\lambda,M_n}^{\alpha-1,0}-\mathcal{I})
\big(\partial_t\big(\kappa(\sigma(\lambda,t),t)\big)\psi(\sigma(\lambda,t),u^n(\sigma(\lambda,t)))\big)\Big)^2\mathrm{d}\lambda\mathrm{d}t,
\\
D_{3,3}=&\int_{\Omega_n}(\dfrac{t-t_{n-1}}{h_n})^{2\alpha}\Big(\int_{\Omega_n} (t_n-\lambda)^{\alpha-1}(\mathcal{I}_{\lambda,M_n}^{\alpha-1,0}-\mathcal{I})\big(\kappa(\sigma(\lambda,t),t)\partial_t\big(\psi(\sigma(\lambda,t),u^n(\sigma(\lambda,t))\big)\big)\Big)^2\mathrm{d}\lambda\mathrm{d}t.
\end{aligned}
\end{equation}
Now, we derive the upper bound for each above terms as follows
\begin{equation}
\begin{aligned}
D_{3,1}\leq& c_\alpha h_n^{-2\alpha}\int_{\Omega_n}(t-t_{n-1})^{2\alpha-1}\Big[\int_{\Omega_n} (t_n-\lambda)^{\alpha-1}\mathrm{d}\lambda \, . \,\int_{\Omega_n} (t_n-\lambda)^{\alpha-1}\Big((\mathcal{I}_{\lambda,M_n}^{\alpha-1,0}-\mathcal{I})\\
&~\times\big(\kappa(\sigma(\lambda,t),t)\psi(\sigma(\lambda,t),u^n(\sigma(\lambda,t)))\big)\Big)^2\mathrm{d}\lambda\Big]\mathrm{d}t\\
\leq&	c_\alpha h_n^{2m-\alpha}(M_n+1)^{-2m}\Vert\psi(.,u^n(.))\Vert^2_{H^m_{\chi^{\alpha-1,0}}(\Omega_n)}\int_{\Omega_n}(t-t_{n-1})^{2\alpha-1}\mathrm{d}t \\
\leq& c_\alpha h_n^{2m+\alpha}(M_n+1)^{-2m}\Vert\psi(.,u^n(.))\Vert^2_{H^m_{\chi^{\alpha-1,0}}(\Omega_n)},
\end{aligned}
\end{equation}
similarly,
\begin{equation}
\begin{aligned}
D_{3,2}
&\leq c {h_n^{-2\alpha}}\int_{\Omega_n}(t-t_{n-1})^{2\alpha}\Big[\int_{\Omega_n} (t_n-\lambda)^{\alpha-1}\mathrm{d}\lambda\, .\, \int_{\Omega_n} (t_n-\lambda)^{\alpha-1}\Big((\mathcal{I}_{\lambda,M_n}^{\alpha-1,0}-\mathcal{I})\\
&~~~\times\big(\psi(\sigma(\lambda,t),u^n(\sigma(\lambda,t)))\big)\Big)^2\mathrm{d}\lambda\Big]\mathrm{d}t\\
&\leq c_\alpha h_n^{2m+\alpha+1}(M_n+1)^{-2m}\Vert\psi(.,u^n(.))\Vert^2_{H^m_{\chi^{\alpha-1,0}}(\Omega_n)},	
\end{aligned}
\end{equation}
and finally,
\begin{equation}\label{d33}
\begin{aligned}
D_{3,3}
&\leq c {h_n^{-2\alpha}}\int_{\Omega_n}(t-t_{n-1})^{2\alpha}\Big[\int_{\Omega_n} (t_n-\lambda)^{\alpha-1}\mathrm{d}\lambda \int_{\Omega_n} (t_n-\lambda)^{\alpha-1}\Big((\mathcal{I}_{\lambda,M_n}^{\alpha-1,0}-\mathcal{I})\\
&~~~\times\big(\partial_t\big(\psi(\sigma(\lambda,t),u^n(\sigma(\lambda,t)))\big)\big)\Big)^2\mathrm{d}\lambda\Big]\mathrm{d}t\\
&\leq c_\alpha {h_n^{\alpha}}\int_{\Omega_n}\Big[ \int_{\Omega_n} (t_n-\lambda)^{\alpha-1}\Big((\mathcal{I}_{\lambda,M_n}^{\alpha-1,0}-\mathcal{I})\big(\partial_t\big(\psi(\sigma(\lambda,t),u^n(\sigma(\lambda,t)))\big)\big)\Big)^2\mathrm{d}\lambda\Big]\mathrm{d}t\\
&\leq c_\alpha {h_n^{2m+\alpha-2}}(M_n+1)^{2-2m}\int_{\Omega_n}\Big[ \int_{\Omega_n} (t_n-\lambda)^{\alpha-1}\big(\partial^{m-1}_\lambda\partial_t\big(\psi(\sigma(\lambda,t),u^n(\sigma(\lambda,t)))\big)\big)^2\mathrm{d}\lambda\Big]\mathrm{d}t\\
& \leq c_\alpha {h_n^{2m+\alpha-1}}(M_n+1)^{2-2m}  \int_{\Omega_n} (t_n-\lambda)^{\alpha-1}\sum_{i=0}^{m}\big(\partial^i_\tau\big(\psi(\tau,u^n(\tau))\big)\big)^2\mathrm{d}\tau\\
& \leq c_\alpha {h_n^{2m+\alpha-1}}(M_n+1)^{2-2m} \Vert\psi(.,u^n(.))\Vert^2_{H^m_{\chi^{\alpha-1,0}}(\Omega_n)}.
\end{aligned}
\end{equation}	
Consequently from Eqs. \eqref{ee}-\eqref{d33}, the upper bound for $ 
E_{1,3} $ can obtain as follows
\begin{equation}\label{e3}
\begin{aligned}
E_{1,3}\leq c_\alpha {h_n^{2m+\alpha+1}}(M_n+1)^{-2m} \Vert\psi(.,u^n(.))\Vert^2_{H^m_{\chi^{\alpha-1,0}}(\Omega_n)}.
\end{aligned}
\end{equation}
Therefore, using Eqs. \eqref{e0}, \eqref{e1}, \eqref{e2}, \eqref{e3} and \eqref{b4} we have
\begin{equation}\label{e3}
\begin{aligned}
\Vert E_{1}\Vert^2\leq c_\alpha h_n^{\alpha} \Big(h_n^{\alpha} \gamma^2 \Vert e_n\Vert^2_{\Omega_n} +h_{n}^{2m+1}(M_n+1)^{-2m}\big(\gamma^2\Vert u^n\Vert^2_{H^m_{\chi^{\alpha-1,0}}(\Omega_n)}+\Vert\psi(.,u^n(.))\Vert^2_{H^m_{\chi^{\alpha-1,0}}(\Omega_n)}\big)\Big).
\end{aligned}
\end{equation} 
\bibliographystyle{acm}
\bibliography{bibli31}
\end{document}